\documentclass[11pt, a4paper, oneside]{amsart}
\usepackage{comment}
\usepackage[english]{babel}
\usepackage{amsmath, amsthm, amsfonts, mathrsfs, amssymb}
\usepackage{mathtools}
\mathtoolsset{centercolon}
\usepackage{booktabs}
\usepackage[shortlabels]{enumitem}
\usepackage[colorlinks, citecolor = blue]{hyperref}
\usepackage{todonotes}
\usepackage{cleveref}
\usepackage{fullpage}
\usepackage[dvipsnames]{xcolor}
\usepackage[T1]{fontenc}

\usepackage{color}
\usepackage{graphicx}

\setlist[itemize]{leftmargin=25pt}
\setlist[enumerate]{leftmargin=25pt}
\newcommand{\Z}{\ensuremath{\mathbb{Z}}}
\newcommand{\Q}{\ensuremath{\mathbb{Q}}}
\newcommand{\R}{\ensuremath{\mathbb{R}}}

\newcommand{\E}{\ensuremath{\mathbb{E}}}
\renewcommand{\P}{\ensuremath{\mathbb{P}}}

\newcommand{\mb}{\mathbf}
\newcommand{\mc}{\mathcal}
\newcommand{\ms}{\mathscr}

\DeclarePairedDelimiter{\cbrace}{\{}{\}}

\DeclarePairedDelimiter{\ip}{\langle}{\rangle}
\DeclarePairedDelimiter{\nrm}{\lVert}{\rVert}

\newcommand{\nrms}[1]{\Bigl\|#1\Bigr\|}

\newcommand{\ips}[1]{\Bigl\langle#1\Bigr\rangle}

\DeclareMathOperator*{\esssup}{ess\,sup}

\newcommand{\dd}{\hspace{2pt}\mathrm{d}}

\newcommand{\ee}{\mathrm{e}}

\def\avint_#1{\mathchoice{\mathop{\kern 0.2em\vrule width 0.6em height 0.69678ex depth -0.58065ex \kern -0.8em \intop}\nolimits_{\kern -0.4em#1}}{\mathop{\kern 0.1em\vrule width 0.5em height 0.69678ex depth -0.60387ex \kern -0.6em \intop}\nolimits_{#1}} {\mathop{\kern 0.1em\vrule width 0.5em height 0.69678ex depth -0.60387ex \kern -0.6em \intop}\nolimits_{#1}} {\mathop{\kern 0.1em\vrule width 0.5em height 0.69678ex depth -0.60387ex \kern -0.6em \intop}\nolimits_{#1}}}
\DeclareMathOperator{\distD}{dist_{\mathscr D}}
\DeclareMathOperator{\Mdyad}{M_{\mathscr D}}
\newcommand{\avg}[2]{\langle #1\rangle_{#2}^{\mu}}
\newcommand{\avgsig}[2]{\langle #1\rangle_{#2}^{\sigma}}
\newcommand{\Acal}{\mathcal A}
\newcommand{\AinfD}{A_\infty^{\mathscr D}(\mu)}
\newcommand{\Chat}{\mathcal C}

\newtheorem{theorem}{Theorem}
\newtheorem{corollary}[theorem]{Corollary}
\newtheorem{lemma}[theorem]{Lemma}
\newtheorem{proposition}[theorem]{Proposition}

\newtheorem{ltheorem}{Theorem}

\newtheorem{lcorollary}[ltheorem]{Corollary}

\theoremstyle{remark}
\newtheorem{remark}[theorem]{Remark}

\theoremstyle{definition}

\numberwithin{theorem}{section}
\numberwithin{equation}{section}

\newcommand{\proofpart}[1]{%
  \par\medskip
  \noindent\textit{#1.}\par\smallskip
}
\allowdisplaybreaks

\title[Sharp mixed $A_p$--$A_\infty$ estimates for sparse operators]
{Sharp mixed $A_p$--$A_\infty$ estimates for sparse operators on filtered and
nonhomogeneous measure spaces}

\author{Francisco Gon\c{c}alves}
\author{Emiel Lorist}

\address[Francisco Gon\c{c}alves and Emiel Lorist]{\hfill\break\indent
Delft Institute of Applied Mathematics \hfill\break\indent
Delft University of Technology \hfill\break\indent
P.O. Box 5031 \hfill\break\indent
2600 GA Delft, The Netherlands}
\email{f.g.j.diasdecarvalho@tudelft.nl}
\email{e.lorist@tudelft.nl}

\thanks{F.G. was supported by the Dutch Research Council (NWO) through grant \href{https://www.nwo.nl/en/projects/vividi223019}{VI.Vidi.223.019}. E.L. was supported by the Dutch Research Council (NWO) through grant \href{https://doi.org/10.61686/ZGRMR99948}{VI.Veni.242.057}.}
\begin{document}

\begin{abstract}
We prove mixed $A_p$--$A_\infty$ estimates for sparse operators in two non-doubling settings.
In the first setting, we consider sparse operators defined using stopping times in continuous time. We obtain both strong- and weak-type bounds with the same powers of the weight characteristics as in the classical setting. Some of our weak-type bounds are even new for the  dyadic filtration on $\R^d$ and, in particular, imply a sharp weak-type $(2,2)$ estimate for Rubio de Francia square functions, solving a problem left open by Garg, Roncal and Shrivastava \cite{GRS21}.

In the second setting, we consider dyadic sparse forms in which distinct cubes may interact, provided their dyadic distance is bounded. We obtain strong-type bounds with the same powers of the weight characteristics as in the classical setting. As a one-dimensional application, we obtain strong-type bounds for Haar shifts over balanced non-doubling measures, answering a quantitative question posed by Conde-Alonso, Pipher, and Wagner \cite{CPW}.
\end{abstract}

\maketitle

\section{Introduction}
Sparse domination is a technique in real-variable harmonic analysis to bound an operator, either pointwise, in norm or in bilinear form, by a positive averaging operator using a sparse family of cubes.
Sparse domination has its roots in Lerner's work \cite{Lerner2010} and became increasingly influential following Hytönen's proof of the $A_2$-theorem \cite{hytonen_sharp_2012}, for which Lerner gave an alternative proof using sparse operators \cite{lerner_simple_2013}.
Since then, sparse domination has developed into a flexible toolbox in harmonic analysis \cite{conde-alonso_pointwise_2016,lerner_pointwise_2015, LernerLoristOmbrosi2022, LernerNazarov2019}, extending far beyond classical Calder\'on--Zygmund theory \cite{bernicot_sharp_2016,CondeAlonsoCuliucDiPlinioOu2017,Lerner2019} and having implications beyond weighted estimates \cite{CuliucDiplinioOu2017, LaukkarinenLorist2026,LoristNieraeth2022,OrtizCaraballoPerezRela2013}.

Classically, sparse operators are built on a dyadic lattice, and each cube interacts only with itself. In this paper, we will relax these two features separately. First, in a continuous-time setting, we replace dyadic generations by stopping times. Second, in the dyadic setting, we allow distinct cubes to interact, using a finite-complexity condition. We will discuss these two generalizations in more detail below, for which we first recall the classical setting.

\smallskip

Let \(\mathscr D\) be a dyadic lattice in \(\mathbb R^d\). We say a collection 
\(\mathcal S\subseteq\mathscr D\) is \emph{\(\eta\)-sparse} if, for every $Q \in \mc{S}$ there is a measurable set $E_Q\subseteq Q$ such that $|E_Q|\ge\eta|Q|$, and the sets $E_Q$ are pairwise disjoint. The classical \emph{sparse operator of type} $r \in (0,\infty)$ is given by
\begin{equation*}
\mathcal A^r_{\mathcal S,\mathscr D}f(x)
:=
\Bigl(\sum_{Q\in\mathcal S}\ip{|f|}_Q^r\mathbf 1_Q(x)\Bigr)^{1/r},
\qquad
\ip{f}_Q:=\frac1{|Q|}\int_Qf\,\mathrm dx.
\end{equation*}
Once sparse domination for an operator is known, its (quantitative) weighted bounds reduce to estimating the corresponding positive sparse operator in terms of the weight characteristics. In the classical setting, this leads to mixed $A_p$--$A_\infty$ bounds, as developed in \cite{hytonen_sharp_2011}. For weights $\sigma,\omega,v$ and $p\in(1,\infty)$, the classical dyadic two weight characteristics \cite{muckenhoupt1972} are
\begin{align*}
[\omega,\sigma]_{A_p}
&:=
\sup_{Q \in \ms{D}}
\ip{\omega}_Q\ip{\sigma}_Q^{p-1},\\
[\omega,v]_{A_1}
&:=
\Bigl\|
\frac{M_{\mathscr D}\omega}{v}
\Bigr\|_{L^\infty}.
\end{align*}
and the dyadic Fujii--Wilson $A_\infty$-characteristic \cite{fuji78,Wilson87} is
\begin{equation*}
[\omega]_{A_\infty}
:=
\sup_{Q \in \ms{D}}
\frac{1}{\omega(Q)}\int_Q M_{\ms{D}}(\mathbf 1_Q\omega),
\end{equation*}
where the dyadic maximal operator above is given by
$
M_{\ms{D}}f
:=
\sup_{Q \in \ms{D}}
\ip{|f|}_Q \mb{1}_Q.
$ 
We note that we need to treat the case $p=1$ separately due to our weight normalization.

Mixed strong-type $A_p$-$A_\infty$-estimates for $\mathcal A^r_{\mathcal S,\mathscr D}$ were established by Lacey and Li \cite{lacey_li_square_2016}. In the form given by Hytönen and Li {\cite[Theorem~1.1]{hytonen_weak_2017}}, these estimates  for $p\in(1,\infty)$ are
\begin{equation}\label{eq:thmhytonenli}
\| \mathcal{A}^r_{\mathcal S,\mathscr D}(\cdot\sigma)\|_{L^p_\sigma\to L^p_\omega}
\lesssim_{p,r,\eta}
[\omega,\sigma]_{A_p}^{1/p}
\bigl([\sigma]_{A_\infty}^{1/p} +
[\omega]_{A_\infty}^{(1/r-1/p)_+}
\bigr),
\end{equation}
where
\(L^p_\sigma:=L^p(\mathbb R^d,\sigma \dd x)\).
For the corresponding weak-type estimate, Hytönen and Li
\cite[Theorem~1.2]{hytonen_weak_2017} showed that,
for \(p\in(1,\infty)\) with \(p\ne r\),
\begin{equation}\label{eq:thmhytonenliweak}
\|\mathcal A^r_{\mathcal S,\mathscr D}(\cdot\sigma)\|
_{L^p_\sigma\to L^{p,\infty}_\omega}
\lesssim_{p,r,\eta}
[\omega,\sigma]_{A_p}^{1/p}
[\omega]_{A_\infty}^{(1/r-1/p)_+}.
\end{equation}
and for $p=r$, the methods of Domingo-Salazar, Lacey and Rey \cite{DomingoSalazarLaceyRey2016} yield
\begin{equation}\label{eq:thmborderline}
\|\mathcal A^r_{\mathcal S,\mathscr D}(\cdot\sigma)\|
_{L^p_\sigma\to L^{p,\infty}_\omega}
\lesssim_{p,r,\eta}
[\omega,\sigma]_{A_p}^{1/p}\bigl({\log(\ee+[\omega]_{A_\infty})}\bigr)^{1/p}.
\end{equation}
For weak-type estimates when $p=1$ we refer to \cite{DomingoSalazarLaceyRey2016,NS26}. Furthermore, we note that Zorin-Kranich \cite{zorin-kranich_a_p-a_infty_2019} extended these mixed \(A_p\)–\(A_\infty\) bounds for sparse operators to abstract dyadic grids in a general non-atomic measure space.

\subsection*{Continuous-time sparse operators}
Continuous-time sparse domination was introduced in
\cite{domelevo_continuous_2025}. We work throughout on a $\sigma$-finite
filtered measure space
\(
(X,\mathcal F,(\mathcal F_t)_{t\in\mathbb R},\mu).
\)
Here $\sigma$-finite means that
$\mu|_{\mathcal F_t}$ is $\sigma$-finite for every $t\in\mathbb R$.
The filtration is indexed by \(\mathbb R\) and need not be discrete or atomic. We write $\E[\cdot\mid\mathcal F_t]$ for conditional expectation and, for a stopping time $\tau$, denote by $\mathcal F_\tau$ the associated stopped
\(\sigma\)-algebra.
For a weight $\sigma$ we write \(\mathrm d\mu_\sigma:=\sigma\,\mathrm d\mu\) and \(L^p_\sigma:=L^p(X,\mu_\sigma)\).

Let $0<\eta<1$, let 
$\mathcal S=\{\tau_j\}_{j=0}^\infty$ be an increasing sequence of stopping times, and let $\nu$ be a measure such that $\nu|_{\mathcal F_{\tau_j}}$ is $\sigma$-finite for every $j\geq 0$. Define the nested sets $E_j:=\{\tau_j<\infty\}$. Then $\mc{S}$ is called $\eta$-sparse with respect to $\nu$ if
\[
\nu(A\cap E_{j+1})\le(1-\eta)\nu(A)
\]
for every $j\ge0$ and every $A\in\mathcal F_{\tau_j}$ with
$A\subseteq E_j$ and $\nu(A)<\infty$. If $\nu=\mu$, we omit it from notation.
This definition recovers the dyadic notion when using the dyadic filtration. Indeed, given a dyadic sparse family, organize its cubes into generations and define $\tau_j(x)$ as the dyadic level of the unique cube in the $j$-th generation containing $x$, with $\tau_j(x)=\infty$ if no such cube exists. The condition above is then precisely the $\eta$-sparseness condition introduced above, and the resulting sparse operators coincide.

For $r>0$, we define
\begin{equation*}
\mathcal A_{\mathcal S}^r(f)
:=
\Bigl(
\sum_{j=0}^\infty
\E[|f|\mid\mathcal F_{\tau_j}]^r\mathbf 1_{E_j}
\Bigr)^{1/r}.
\end{equation*}
For $r=1$, write $\mathcal A_{\mathcal S}:=\mathcal A_{\mathcal S}^1$.
Continuous-time analogues of the classical $A_p$, $A_1$ and $A_\infty$  characteristics are
\begin{align*}
[\omega,\sigma]_{A_p}
&:=
\sup_\tau
\bigl\|\mathbf 1_{\{\tau<\infty\}}
\E[\omega\mid\mathcal F_\tau]
\E[\sigma\mid\mathcal F_\tau]^{p-1}
\bigr\|_{L^\infty_\mu},\qquad {p\in (1,\infty)}
\\
[\omega,v]_{A_1}
&:=
\sup_\tau
\Bigl\|
\mathbf1_{\{\tau<\infty\}}
\frac{\E[\omega\mid\mathcal F_\tau]}{v}
\Bigr\|_{L^\infty_\mu},\\
[\omega]_{A_\infty}
&:=
\sup_\tau
\Bigl\|\mathbf 1_{\{\tau<\infty\}}
\frac{\E[M_\tau\omega\mid\mathcal F_\tau]}
{\E[\omega\mid\mathcal F_\tau]}
\Bigr\|_{L^\infty_\mu}.
\end{align*}
where the suprema range over stopping times and
\begin{equation*}
M_\tau f:=\esssup_{\rho\ge\tau}\mathbf 1_{\{\rho<\infty\}}
\E[|f|\mid\mathcal F_\rho].
\end{equation*}

Our first main result shows that the mixed estimate in \eqref{eq:thmhytonenli} also holds in continuous time, with the same powers as in the classical dyadic result.

\begin{ltheorem}[Continuous-time strong \(A_p\)--\(A_\infty\) estimate]
\label{thm: sparse operators ap ainfty bound}
Let $1<p<\infty$, $r>0$, $0<\eta<1$, and let
$\mathcal S=\{\tau_j\}_{j=0}^\infty$ be an $\eta$-sparse sequence of stopping
times. For weights $\sigma,\omega$, we have
\begin{equation*}
\|\mathcal A^r_{\mathcal S}(\cdot\sigma)\|_{L^p_\sigma\to L^p_\omega}
\lesssim_{p,r,\eta}
[\omega,\sigma]_{A_p}^{1/p}
\left(
[\sigma]_{A_\infty}^{1/p}
+[\omega]_{A_\infty}^{(1/r-1/p)_+}
\right).
\end{equation*}
\end{ltheorem}
For \(r=1\), Theorem \ref{thm: sparse operators ap ainfty bound} extends the one weight estimate in \cite[Theorem~4]{domelevo_continuous_2025} to a two weight mixed \(A_p\)--\(A_\infty\) bound. Combined with the sparse domination result in \cite[Theorem~1.1]{ChenZhangZhang2026}, the case \(r=1\)  gives mixed two weight bounds for discrete-time martingale transforms and their maximal functions.

The next theorem gives weak-type estimates, including the two weight \(p=1\) endpoint, which we need to treat separately due to our weight normalization

\begin{ltheorem}[Continuous-time weak \(A_p\)--\(A_\infty\) estimates]
\label{thm: sparse operators weak type bound}
Let \(r>0\), \(0<\eta<1\), and let
\(\mathcal S=\{\tau_j\}_{j=0}^\infty\) be an
\(\eta\)-sparse sequence of stopping times.
\begin{enumerate}[label=\normalfont(\roman*)]
\item\label{item: sparse operators weak type bound p not 1}
Let \(1<p<\infty\), and let \(\sigma,\omega\) be weights. Then
\[
\|\mathcal A_{\mathcal S}^r(\cdot\sigma)\|
_{L^p_\sigma\to L^{p,\infty}_\omega}
\lesssim_{p,r,\eta}
[\omega,\sigma]_{A_p}^{1/p}
\begin{cases}
[\omega]_{A_\infty}^{1/r-1/p}, & 0<r<p,\\[1mm]
\bigl(\log(e+[\omega]_{A_\infty})\bigr)^{1/p},
    & r=p,\\[1mm]
1, & r>p.
\end{cases}
\]

\item\label{thm:sparse-weak-endpoint}
Let \(v,\omega\) be weights. Then
\[
\|\mathcal A_{\mathcal S}^r\|
_{L^1_v\to L^{1,\infty}_\omega}
\lesssim_{r,\eta}
[\omega,v]_{A_1}
\begin{cases}
[\omega]_{A_\infty}^{1/r-1}, & 0<r<1,\\[1mm]
\log(e+[\omega]_{A_\infty}), & r=1,\\[1mm]
1, & r>1.
\end{cases}
\]
\end{enumerate}
\end{ltheorem}

To the best of our knowledge, Theorem \ref{thm: sparse operators weak type bound} contains the first weak-type bounds for sparse operators on continuous-time filtered measure spaces under the corresponding \(A_p\)–\(A_\infty\) assumptions. When specialized to the dyadic filtration on $\R^d$, part~\ref{item: sparse operators weak type bound p not 1} recovers \eqref{eq:thmhytonenliweak} and \eqref{eq:thmborderline}. We note that the logarithm in the case $p=r=2$ was very recently shown to be sharp by Os\k{e}kowski \cite[Theorem~1.2]{Osekowski2026}, since the dyadic square function admits pointwise sparse domination by \(\mathcal A^2_{\mathcal S}\).

The case $r\geq 1$ in Theorem \ref{thm: sparse operators weak type bound}\ref{thm:sparse-weak-endpoint} in the dyadic setting  can, for example, be found in \cite[Corollary~C]{NS26}. As shown by Lerner, Nazarov and Ombrosi \cite{LernerNazarovOmbrosi2020}, the logarithm in the case $r=1$ is sharp.
Perhaps most notably, the case $r<1$ is even new in the dyadic setting, and solves the open problem stated by Nieraeth and Stockdale after \cite[Corollary~C]{NS26}.
Taking \(r=\tfrac1{p_0}\) and rescaling, Theorem \ref{thm: sparse operators weak type bound} removes the logarithmic loss in the endpoint estimate of Frey and Nieraeth \cite[Theorem~1.4]{FN19} when \(p_0>1\) and \(q_0=\infty\).

In particular, the case \(r=\tfrac12\), combined with the pointwise sparse domination in \cite[Theorem 1.3]{GRS21}, improves the best known two weight weak-type \((2,2)\) bound for the Rubio de Francia square function. For a collection \(\mc I\) of pairwise disjoint intervals in \(\mathbb R\), define
\[
R_{\mc I}f
:=
\Bigl(
\sum_{I\in\mc I}
\bigl|\mathcal F^{-1}(\mathbf1_I\mathcal F(f))\bigr|^2
\Bigr)^{1/2},
\]
where \(\mathcal F\) denotes the Fourier transform. In the following corollary, the suprema defining the weight characteristics are taken over all intervals, not just dyadic ones.
\begin{lcorollary}[Weak-type bound for the Rubio de Francia square function]
\label{cor:rubio-weak-endpoint}
Let \(v,\omega\) be weights on \(\mathbb R\), and let
\(\mc I\) be a collection of pairwise disjoint frequency intervals.
Then
\[
\|R_{\mc I}f\|_{L^{2,\infty}_\omega}
\lesssim
[\omega,v]_{A_1}^{1/2}
[\omega]_{A_\infty}^{1/2}
\|f\|_{L^2_v}.
\]
\end{lcorollary}
In the one weight setting, Corollary \ref{cor:rubio-weak-endpoint} answers a question raised by Garg, Roncal and Shrivastava. Indeed, it removes the logarithmic loss in  \cite[Corollary 1.6]{GRS21}, yielding the sharp dependence on the weight characteristic (see \cite[Section 1.4]{DFPR25}). A similar improvement follows for, e.g., \cite[Theorem~A and Proposition 1.2]{DFPR25}.

\smallskip

Let \(M\) denote Doob's maximal operator
\[
Mf:=\esssup_{t\in\mathbb R}\E[|f|\mid\mathcal F_t].
\]
Theorems~\ref{thm: sparse operators ap ainfty bound} and \ref{thm: sparse operators weak type bound} give the following bounds for \(M\). The strong-type estimate refines the \(A_p\) bound in
\cite[Theorem~5]{domelevo_continuous_2025} to a mixed
\(A_p\)--\(A_\infty\) estimate, whereas the weak-type estimates seem new in this generality.
\begin{lcorollary}[Mixed bounds for Doob's maximal operator]
\label{cor: max op ap ainfty bound}
Let \(\sigma,\omega\) be weights. For \(1<p<\infty\),
\[
\begin{aligned}
\|M(\cdot\sigma)\|_{L^p_\sigma\to L^{p,\infty}_\omega}
&\lesssim_p [\omega,\sigma]_{A_p}^{1/p},\\
\|M(\cdot\sigma)\|_{L^p_\sigma\to L^p_\omega}
&\lesssim_p [\omega,\sigma]_{A_p}^{1/p}
[\sigma]_{A_\infty}^{1/p}.
\end{aligned}
\]
Moreover, for every weight \(v\),
\[
\|M\|_{L^1_v\to L^{1,\infty}_\omega}
\lesssim [\omega,v]_{A_1}.
\]
\end{lcorollary}
Our continuous-time strong-type estimate builds on a sequential two weight \(T1\) theorem for positive operators. 
To state it, let us introduce some terminology.  
A \emph{stopping region} is a pair $(\tau,A)$, where $\tau$ is a stopping time and $A\in\mathcal F_\tau$. Let $\{\tau_j\}_{j=0}^\infty$ be a sparse sequence of stopping times with respect to $\mu_\sigma$, and set $E_j:=\{\tau_j<\infty\}$. A doubly indexed family \(\mathcal P=\{A_{j,m}\}_{j,m}\) is called \emph{subordinate} to \(\{\tau_j\}_{j=0}^\infty\) if $A_{j,m}\in\mathcal F_{\tau_j}$, $A_{j,m}\subseteq E_j$, and, for each fixed $j$, the sets $\{A_{j,m}\}_m$ are pairwise disjoint.
For each \(j\), the sets \(\{A_{j,m}\}_m\) select disjoint pieces of \(E_j\) that are known at time \(\tau_j\). In the dyadic setting, each piece would be a union of stopping cubes.

Let $\{T_k\}_{k=0}^\infty$ be an increasing sequence of stopping times, and let each $\lambda_k$ be a nonnegative $\mathcal F_{T_k}$-measurable function satisfying $\lambda_k\mathbf 1_{\{T_k=\infty\}}=0.$
Define
\begin{equation}\label{eq: stochastic T}
T(f):=\sum_{k=0}^\infty\lambda_k\E[f\mid\mathcal F_{T_k}],
\end{equation}
and, for a stopping time $\tau$, define the tail operator
\begin{equation*}
T_\tau(f)
:=
\sum_{k=0}^\infty
\lambda_k\E[f\mid\mathcal F_{T_k}]\mathbf 1_{\{T_k\ge\tau\}}.
\end{equation*}

\begin{ltheorem}[Sequential two weight \(T1\) theorem]
\label{thm: stoch t1 seq test}
Let \(T\) be the operator in \eqref{eq: stochastic T}. For
$1<p,q<\infty$ and weights $\sigma,\omega$, we have
\[
\|T(\cdot\sigma)\|_{L^p_\sigma\to L^q_\omega}
\simeq_{p,q}
\mathfrak T_{p,q}+\mathfrak T^*_{p,q},
\]
where $\theta\in(1,\infty]$ is defined by
\(\frac1\theta=(\frac1q-\frac1p)_+\), and
\begin{align*}
\mathfrak T_{p,q}
&:=
\sup_{\substack{\{\tau_j\}_{j=0}^\infty,\mathcal P}}
\left\|
\left\{
\frac{
\|\mathbf 1_{A_{j,m}}T_{\tau_j}(\sigma)\|_{L^q_\omega}
}{
\sigma(A_{j,m})^{1/p}
}
\right\}_{j,m}
\right\|_{\ell^\theta},
\\
\mathfrak T^*_{p,q}
&:=
\sup_{\substack{\{\upsilon_i\}_{i=0}^\infty,\mathcal Q}}
\left\|
\left\{
\frac{
\|\mathbf 1_{B_{i,m}}T_{\upsilon_i}(\omega)\|_{L^{p'}_\sigma}
}{
\omega(B_{i,m})^{1/q'}
}
\right\}_{i,m}
\right\|_{\ell^\theta}.
\end{align*}
Here \(\{\tau_j\}_{j=0}^\infty\) and \(\{\upsilon_i\}_{i=0}^\infty\) range over $\tfrac12$-sparse stopping time sequences with respect to \(\mu_\sigma\) and \(\mu_\omega\) respectively, and $\mathcal P$ and \(\mathcal Q\) range over families subordinate to the corresponding sequences.
\end{ltheorem}
We will use Theorem~\ref{thm: stoch t1 seq test} only
with \(p=q\) to prove
Theorem~\ref{thm: sparse operators ap ainfty bound}.
When \(p\le q\), and thus $\theta=\infty$, the testing constants reduce to
suprema over individual stopping regions.
When \(q<p\), they require \(\ell^\theta\) bounds
over families subordinate to sparse stopping sequences.
Moreover, the range \(p<q\) imposes a structural restriction: boundedness forces each coefficient \(\lambda_k\) to vanish on the nonatomic part of the stopped measure space \((X,\mathcal F_{T_k},\mu)\), as shown in Proposition~\ref{prop:atomicity-p-less-q}. Therefore, \(T=0\) if all these spaces are nonatomic. For \(q\le p\), nonzero operators can exist even on nonatomic filtrations.

The dyadic analogue of Theorem~\ref{thm: stoch t1 seq test} for all \(1<p,q<\infty\) is due to H\"anninen, Hyt\"onen, and Li \cite{hanninen_two-weight_2016}. Their sequential testing theorem unifies the earlier
results for \(p\le q\) due to Lacey, Sawyer, and
Uriarte-Tuero \cite{lacey_two_2010} and, in a different
formulation, for \(q<p\) due to Tanaka \cite{tanaka2014}.

In the filtered measure space setting, for \(p\le q\), Tanaka and Terasawa treated the cases \((\sigma,\omega)=(1,\omega)\) and \((\sigma,\omega)=(w,1)\) for discrete filtrations under the so-called dyadic logarithmic bounded oscillation condition (DLBO) \cite{tanaka_positive_2013}, i.e.
\begin{equation*}
\sum_{k\ge j}\E[\lambda_k\mid\mathcal F_{T_j}]\simeq\sum_{k\ge j}\lambda_k,
\end{equation*}
which is quite restrictive, see also \cite{CO2009,COV2004,COV2006}. 
Under this DLBO condition, they also characterized boundedness for \(1<q<p<\infty\) in the case \(\sigma=1\) using a discrete
Wolff potential \cite[Theorem~1.2]{tanaka_positive_2013}.
In a later paper \cite{tanaka_characterization_2013}, they treated the unweighted case \(\sigma=\omega=1\) without this condition. Moreover, Chen, Zhu, Zuo, and Jiao \cite[Theorem~1.1]{chen_two-weighted_2020} considered the case \(p\le q\), assuming that both weights belong to \(A_1\). Theorem~\ref{thm: stoch t1 seq test} removes all aforementioned additional assumptions.

\smallskip

For our weak-type estimates in Theorem \ref{thm: sparse operators weak type bound}, in the case \(1<p<\infty\) and \(r<p\), we could use the weak-type testing theorem of Chen, Zhu, Zuo, and Jiao \cite[Theorem~1.6]{chen_two-weighted_2020} on a sampled filtration. However, we prefer to give a direct proof for all \(r>0\), following ideas of Lacey and Scurry \cite[Section 2]{lacey_scurry_2012}, further developed by Domingo-Salazar, Lacey, and Rey \cite[Section~4]{DomingoSalazarLaceyRey2016}  and 
Hyt\"onen and Li \cite[Section~4]{hytonen_weak_2017}. The main ingredient we additionally add is the weighted counting estimate in Lemma~\ref{lem:weak-counting}. This is particularly useful when \(r\le p\) and in particular enables the removal of a logarithmic term when $r<p$, leading to Corollary \ref{cor:rubio-weak-endpoint}.

\subsection*{Finite-complexity dyadic sparse forms} 
Finite-complexity dyadic operators, in particular Haar shifts, are fundamental model operators in modern Calderón--Zygmund theory. The study of Haar shifts beyond the doubling setting led to the introduction of the regularity condition \cite{lopez-sanchez_dyadic_2014}, now known as balancedness. Sparse domination in related non-doubling settings was subsequently developed in \cite{conde-alonso_nondoubling_2019,volberg_sparse_2018}.
For doubling measures, the complexity of a Haar shift can typically be absorbed into the constant in the usual sparse bound. By contrast, for balanced non-doubling measures on $\R$, Conde-Alonso, Pipher, and Wagner \cite{CPW} showed that such a bound can fail for shifts of positive complexity. 
They replaced the classical form by one that allows different intervals to interact, with the complexity controlling their dyadic distance. Their result gave qualitative weighted bounds and a complexity-dependent $\Acal_p$ condition. However, the sharp dependence on the weight characteristics, i.e. the analogue of the $\Acal_2$-theorem in this setting, remained open \cite[Remark~3.4]{CPW}. Subsequent developments include endpoint estimates, commutator theory, and vector- and matrix-weighted extensions \cite{BBDPW26, BCAPW25,CAW25}.

We introduce a positive sparse form that is inspired by the construction in \cite{CPW}. It allows cubes at bounded dyadic distance to interact, but is not tied to a particular Haar shift. Let \(\mu\) be a locally finite Borel measure on \(\mathbb R^d\), and let \(\mathscr D\) be a dyadic lattice on \(\mathbb R^d\). For $Q \in \ms{D}$ we let $\ip{f}_Q^\mu$ be the average of $f$ with respect to $\mu$ and we replace the Lebesgue measure by $\mu$ in the definition of sparseness. We omit cubes of zero \(\mu\)-measure from all averages, sums, and suprema.
We write \(\distD(Q,P)\) for the smallest number of steps from \(Q\) to \(P\), moving at each step to a dyadic parent or child. A relation \(\Gamma\subseteq\mathscr D\times\mathscr D\) has \emph{complexity} at most \(\kappa\in\mathbb N_0\) if \(\distD(Q,P)\le\kappa\) for every \((Q,P)\in\Gamma\). For an $\eta$-sparse family \(\mathcal S\subseteq\mathscr D\) we define
\[
\Chat_{\mathcal S,\Gamma}(f,g)
:=
\sum_{\substack{Q,P\in\mathcal S\\(Q,P)\in\Gamma}}
\avg{f}{Q}\avg{g}{P}\,c_{Q,P},
\qquad
c_{Q,P}:=\min\{\mu(Q),\mu(P)\}.
\]
The sparse form $\Chat_{\mathcal S,\Gamma}$ allows diagonal, nested, and separated interactions. When \(\Gamma=\{(Q,Q):Q\in\mathscr D\}\), it reduces to the classical sparse form.
Note that, since \(f\) is averaged over \(Q\) and \(g\) over \(P\), the order of the pair matters.

The two weight characteristic related to the form $\Chat_{\mathcal S,\Gamma}$ is as follows. For \(1<p<\infty\) and weights \(\sigma, \omega\), we define
\[
[\omega,\sigma]_{\Acal_{p,\Gamma}(\mu)}
:=
\sup_{(Q,P)\in\Gamma}
\frac{c_{Q,P}^p}{\mu(Q)\mu(P)^{p-1}}
\avg{\omega}{P}\bigl(\avg{\sigma}{Q}\bigr)^{p-1}.
\]
For the diagonal relation, this is precisely the usual dyadic two weight \(A_p\) characteristic:
\[
[\omega,\sigma]_{A_p^{\ms{D}}(\mu)}
:=
\sup_{Q\in\mathscr D}
\avg{\omega}{Q}\bigl(\avg{\sigma}{Q}\bigr)^{p-1}.
\]
For a weight \(v\), set
\[
[v]_{\AinfD}:=
\sup_{Q\in\mathscr D}
\frac{1}{v(Q)}\int_Q\Mdyad(v\mathbf 1_Q)\,\mathrm d\mu, \qquad \Mdyad f(x):=\sup_{\substack{Q\in\mathscr D}}\avg{|f|}{Q}\mb{1}_Q.
\]

Our main theorem in this setting shows that finite-complexity interactions preserve the classical mixed \(A_p\)--\(A_\infty\) dependence. In particular, the powers of the weight characteristics are the same as for the classical sparse form, and hence as in the case \(r=1\) of Theorem~\ref{thm: sparse operators ap ainfty bound}. These powers are sharp in general, since the classical diagonal relation is a special case.
\begin{ltheorem}[Finite-complexity dyadic mixed estimate]
\label{thm: finite complexity mixed intro}
Let $0<\eta<1$, let $\mathcal S\subseteq\mathscr D$ be $\eta$-sparse, let
$\Gamma\subseteq\mathscr D\times\mathscr D$ be a relation of complexity at most $\kappa$, let $1<p<\infty$, and let \(\sigma,\omega\) be weights. Then, for
all \(f\in L^p_\sigma\) and
\(g\in L^{p'}_\omega\),
\[
\bigl|\Chat_{\mathcal S,\Gamma}(f\sigma,g\omega)\bigr|
\lesssim_{p,d,\kappa,\eta}
[\omega,\sigma]_{\Acal_{p,\Gamma}(\mu)}^{1/p}
\Bigl(
[\sigma]_{\AinfD}^{1/p}
+[\omega]_{\AinfD}^{1/p'}
\Bigr)
\|f\|_{L^p_\sigma}
\|g\|_{L^{p'}_\omega}.
\]
\end{ltheorem}
Theorem \ref{thm: finite complexity mixed intro} does not require \(\mu\) to be balanced or the weights \(\sigma\) and \(\omega\) to be related. 
Still, if the relation \(\Gamma\) contains all diagonal pairs, setting \(\sigma=\omega^{1-p'}\) and estimating the $A_\infty$ characteristics by the $\Acal_{p,\Gamma}(\mu)$ characteristic gives the corresponding one weight estimate, with the familiar power \(\max\{1,1/(p-1)\}\), see Subsection~\ref{sec:one-weight-diagonal-consequence}.
For the diagonal relation and Lebesgue measure, the theorem reduces to the case \(r=1\) of \eqref{eq:thmhytonenli}.

The proof works by splitting the relation into finitely many injective branches. On each branch, a stopping-time argument gives a testing estimate. The $\Acal_{p,\Gamma}(\mu)$-characteristic bounds the individual terms, while packing estimates bound their sum. Balancedness enters when we compare our form with the \cite{CPW} form, which we will describe next.

\subsection*{Application to balanced non-doubling Haar shifts}
For the application to Haar shifts as in \cite{CPW}, we now specialize to \(d=1\). Set
\[
m_\mu(I):=\frac{\mu(I_-)\mu(I_+)}{\mu(I)}.
\]

An atomless measure \(\mu\) is called balanced if \(m_\mu(I)\simeq m_\mu(I^{(1)})\) uniformly over dyadic intervals, where \(I^{(1)}\) is the dyadic parent of \(I\). In \cite{CPW} it was shown that a Haar shift of total complexity at most \(N\) is dominated by two sparse forms: the classical diagonal form and a separated form that pairs intervals at dyadic distance at most \(N+2\). When \(\mu\) is balanced, the coefficient in the separated form is comparable to \(c_{I,J}\). Moreover, our $\Acal_{p,\Gamma}(\mu)$-characteristic is bounded by the weight characteristic from \cite{CPW} defined below. We prove these two comparisons in Section~\ref{sec:cpw-comparison}.
For \(1<p<\infty\), put \(\sigma=\omega^{1-p'}\) and define the
complexity-dependent characteristic
\begin{equation*}
[\omega]_{A_p^N(\mu)}
:=
\max\Biggl\{
\sup_{I\in\mathscr D}
\avg{\omega}{I}\bigl(\avg{\sigma}{I}\bigr)^{p-1},
\sup_{\substack{I,J\in\mathscr D,\ I\neq J\\
\distD(I,J)\le N+2}}
\frac{m_\mu(I)^{p/2}m_\mu(J)^{p/2}}
{\mu(J)\mu(I)^{p-1}}
\avg{\omega}{I}\bigl(\avg{\sigma}{J}\bigr)^{p-1}
\Biggr\}.
\end{equation*}
Together with the mentioned sparse domination theorem, these comparisons allow us to pass from our dyadic estimate to the following bound for Haar shifts.

\begin{lcorollary}[Sharp mixed bounds for Haar shifts]
\label{cor:cpw-sharp-quantitative}
Let \(N\in\mathbb N_0\), \(1<p<\infty\), let \(\mu\) be balanced and let \(T\) be a Haar shift with coefficients bounded by \(1\) and
total complexity at most \(N\). Let \(\omega\) be a weight and set \(\sigma=\omega^{1-p'}\). Then
\[
\|T\|_{L^p_\omega\to L^p_\omega}
\lesssim_{p,N,\mu}
[\omega]_{A_p^N(\mu)}^{1/p}
\left(
[\sigma]_{\AinfD}^{1/p}
+[\omega]_{\AinfD}^{1/p'}
\right),
\]
and, in particular,
\[
\|T\|_{L^p_\omega\to L^p_\omega}
\lesssim_{p,N,\mu}
[\omega]_{A_p^N(\mu)}^{\max\{1,1/(p-1)\}}.
\]
\end{lcorollary}

Corollary~\ref{cor:cpw-sharp-quantitative} can be seen as the $A_2$-theorem for Haar shifts in the balanced setting and answers the question posed in \cite[Remark~3.4]{CPW}.
The optimal dependence on \(N\) remains open.

\subsection*{Organization of the paper}
In Section~\ref{sec:preliminaries}, we collect the preliminaries for the continuous-time setting. We prove Theorem~\ref{thm: stoch t1 seq test} in Section~\ref{sec:continuous-time-t1}.
Sections~\ref{sec:sparse-testing} and~\ref{sec:testing-to-characteristics} contain the testing reduction and the estimates for the testing constants, completing the proof of Theorem~\ref{thm: sparse operators ap ainfty bound}.
In Section~\ref{sec:weak-type}, we give a proof of Theorem~\ref{thm: sparse operators weak type bound}. Section~\ref{sec:applications} contains the applications of Theorems~\ref{thm: sparse operators ap ainfty bound} and \ref{thm: sparse operators weak type bound} to the Rubio de Francia square function and Doob's maximal operator, i.e. Corollaries \ref{cor:rubio-weak-endpoint} and \ref{cor: max op ap ainfty bound}.

We turn to the finite-complexity dyadic setting in Section~\ref{sec:finite-complexity}, where we establish a branch decomposition and the testing theorem on a branch.
In Section~\ref{sec:finite-complexity-characteristics}, we prove the mixed estimates and deduce Theorem~\ref{thm: finite complexity mixed intro}.
Finally, in Section~\ref{sec:cpw-comparison}, we compare our sparse
form and weight characteristic with those of \cite{CPW} and prove
Corollary~\ref{cor:cpw-sharp-quantitative} for Haar shifts over
balanced measures.
\section{Continuous-time preliminaries}\label{sec:preliminaries}

\subsection{Weighted conditional expectations and sequential maximal operators}
In Sections~\ref{sec:preliminaries}--\ref{sec:weak-type},
we work on a \(\sigma\)-finite filtered measure space
\[
\bigl(X,\mathcal F,(\mathcal F_t)_{t\in\mathbb R},\mu\bigr).
\]
A weight is a measurable function \(\omega\colon X \to (0,\infty)\) such that the measure
\(
\mathrm d\mu_\omega:=\omega\,\mathrm d\mu
\)
is \(\sigma\)-finite on every \(\mathcal F_t\). We write \(L^p_\omega:=L^p(X,\mu_\omega)\). {For a weight \(\sigma\) and a measurable set \(A\) we write \(\sigma(A):=\mu_\sigma(A)\).} For every stopping time \(\tau \colon X \to \R\cup \cbrace{\infty}\), both \(\mu|_{\mathcal F_\tau}\) and \(\mu_\omega|_{\mathcal F_\tau}\) are \(\sigma\)-finite. If \(\tau_1\) and \(\tau_2\) are stopping times and \(A\in\mathcal F_{\tau_1}\), then we will often use the fact that
\begin{equation}\label{eq:stopped-sigma-algebra-facts}
A\cap\{\tau_1\le \tau_2\},\,A\cap\{\tau_1<\tau_2\}\in\mathcal F_{\tau_2},
\end{equation}
see \cite[Theorem 3.1.13 and 6.1.4]{CohenElliott2015}. Moreover, if \(f\) is \(\mathcal F_{\tau_1}\)-measurable, then \(f\mathbf 1_{\{\tau_1\le\tau_2\}}\) is \(\mathcal F_{\tau_2}\)-measurable, as can be checked by considering
level sets. Throughout, we extend every sequence \(\{\tau_j\}_{j\ge0}\) of stopping times by the convention \(\tau_\infty:=\infty\).

Let \(\mathcal G\subset\mathcal F\) be a sub-\(\sigma\)-algebra such that
\(\mu|_{\mathcal G}\) and \(\mu_\omega|_{\mathcal G}\) are
\(\sigma\)-finite. We denote conditional expectation with respect to
\(\mu_\omega\) by \(\E_\omega[\cdot\mid\mathcal G]\). For every nonnegative
measurable function \(f\), the identity
\begin{equation}\label{eq: weighted cond exp identity}
\E[f\omega\mid\mathcal G]
=
\E_\omega[f\mid\mathcal G]\E[\omega\mid\mathcal G]
\end{equation}
 holds \(\mu\)-almost everywhere. 




Let $\mathcal S=\{\tau_j\}_{j=0}^\infty$ be an increasing sequence of
stopping times, and set $E_j:=\{\tau_j<\infty\}$. 
For a weight $\omega$, define
\begin{equation*}
M_{\omega,\mathcal S}f
:=
\sup_{j\ge0}\mathbf 1_{E_j}\E_\omega[|f|\mid\mathcal F_{\tau_j}].
\end{equation*}
\begin{proposition}[Sequential weighted Doob inequality]\label{prop: sequential weighted doob}
Let $\omega$ be a weight, let $p\in(1,\infty)$, and let
$\mathcal S=\{\tau_j\}_{j=0}^\infty$ be an increasing sequence of stopping
times. Then, for every $f\in L^p_\omega$,
\begin{equation*}
\|M_{\omega,\mathcal S} f\|_{L^p_\omega}\le p'\|f\|_{L^p_\omega}.
\end{equation*}
\end{proposition}

\begin{proof}
Since $\tau_j\le \tau_{j+1}$, the \(\sigma\)-algebras
$\mathcal F_{\tau_j}$ form a filtration. Apply Doob's $L^p$ maximal inequality on the countable filtered measure space $(X,\mathcal F,(\mathcal F_{\tau_j})_{j\ge0},\mu_\omega)$
\cite[Theorem~3.2]{riederer_refined_2019} to the martingale $\E_\omega[|f|\mid\mathcal F_{\tau_j}]$. The indicators $\mathbf 1_{E_j}$ can only decrease the supremum.
\end{proof}

\subsection{Sparse stopping and Carleson embedding}
The following lemma is the probabilistic version of the well-known principal stopping cube lemma. It selects principal stopping times from a sequence of martingale averages. Between two consecutive principal indices, each average is at most twice the current principal average. Let us note that $\eta$-sparseness with respect to a measure $\nu$ can be equivalently characterized by 
\begin{equation}\label{eq: sparse property prelim}
\E_\nu[
\mathbf 1_{E_j\setminus E_{j+1}}
\mid\mathcal F_{\tau_j}]
\ge
\eta\mathbf 1_{E_j},
\end{equation}
cf. \cite[Section 1]{CondeAlonsoLoristRey2026}.

\begin{lemma}
\label{lem:weighted-principal-stopping-sequence}
Let $\{T_k\}_{k=0}^\infty$ be an increasing sequence of stopping times, let $\sigma$ be a weight, {let \(1\le p<\infty\)} and let $f\in L^{{p}}_\sigma$ be nonnegative. Write
\(
f_k:=\E_\sigma[f\mid\mathcal F_{T_k}],
\)
and define principal indices by
\[
N_0:=\inf\{k\ge0:f_k>0\},
\qquad
N_{j+1}:=
\begin{cases}
\inf\{k>N_j:f_k>2f_{N_j}\}, & N_j<\infty\\
\infty, & N_j=\infty
\end{cases},
\]
which is a stopping time for the discrete
filtration $\{\mathcal F_{T_k}\}_{k=0}^\infty$. Then
\begin{align*}
 f_k&\le2f_{N_j}
&&\text{on }\{N_j\le k<N_{j+1}\},
\end{align*}
Moreover, $\{T_{N_j}\}_{j=0}^\infty$ is a $\tfrac12$-sparse sequence of stopping times with respect to
$\mu_\sigma$.
\end{lemma}

\begin{proof}
Since the \(T_k\) are increasing, \(\{\mathcal F_{T_k}\}_{k=0}^\infty\) is a discrete filtration.
Since \(\mu_\sigma|_{\mathcal F_{T_k}}\) is \(\sigma\)-finite and \(f\in L^p_\sigma\), each \(f_k\) is well-defined in \(L^p_\sigma\) by \cite[Corollary~2.6.30]{hytonen_analysis_2016}.
The definitions of \(N_0\) and \(N_{j+1}\) therefore show inductively that each \(N_j\) is a stopping time for this filtration. They also give
$
f_k\le 2f_{N_j}$ on $\{N_j\le k<N_{j+1}\}.$
Define $\rho_j:=T_{N_j}$, then
\[
\{\rho_j\le t\}
=
\bigcup_{k=0}^\infty
\{N_j=k\}\cap\{T_k\le t\}\in \mc{F}_t,
\]
so \(\rho_j\) is a stopping time. 

It remains to prove sparsity. On \(\{\rho_{j+1}<\infty\}\), we have
\(
f_{N_{j+1}}>2f_{N_j}.
\)
Moreover, since
\(\rho_j\le\rho_{j+1}\) and
\(\{\rho_{j+1}<\infty\}\in\mathcal F_{\rho_{j+1}}\),
the stopped conditional-expectation identities and the tower property give
\[
\begin{aligned}
\E_\sigma[
f_{N_{j+1}}\mathbf 1_{\{\rho_{j+1}<\infty\}}
\mid\mathcal F_{\rho_j}]
&=
\E_\sigma[
f\mathbf 1_{\{\rho_{j+1}<\infty\}}
\mid\mathcal F_{\rho_j}]
\le
f_{N_j}\mathbf 1_{\{\rho_j<\infty\}}.
\end{aligned}
\]
Since \(f_{N_j}>0\) on \(\{\rho_j<\infty\}\), it follows that
\[
\E_\sigma[
\mathbf 1_{\{\rho_{j+1}<\infty\}}
\mid\mathcal F_{\rho_j}]
\le
\frac{\mathbf 1_{\{\rho_j<\infty\}}}{2f_{N_j}}
\E_\sigma[
f_{N_{j+1}}\mathbf 1_{\{\rho_{j+1}<\infty\}}
\mid\mathcal F_{\rho_j}]
\le
\frac12\mathbf 1_{\{\rho_j<\infty\}}.
\]
Thus
\(\{\rho_j\}_{j=0}^\infty\) is \(\tfrac12\)-sparse with respect to
\(\mu_\sigma\) by \eqref{eq: sparse property prelim}.
\end{proof}

The following is the martingale Carleson embedding theorem \cite[Theorem~3.1 and Corollary~3.4]{tanaka_positive_2013}, applied to the sampled filtration $(\mathcal F_{\tau_j})_{j\ge0}$. We include the short proof to record the precise constant.
\begin{lemma}[Sparse stopping-time embedding]\label{lemma: carleson embedding}
Let \(\sigma\) be a weight, let
\(\{\tau_j\}_{j=0}^\infty\) be an \(\eta\)-sparse sequence of stopping times
with respect to \(\mu_\sigma\), and set
\(E_j:=\{\tau_j<\infty\}\). Then, for every \(p\in(1,\infty)\) and every
nonnegative \(f\in L^p_\sigma\),
\begin{equation*}
\biggl(
\sum_{j=0}^\infty
\int_{E_j}
\E_\sigma[f\mid\mathcal F_{\tau_j}]^p
\,\mathrm d\mu_\sigma
\biggr)^{1/p}
\le
\eta^{-1/p}p'\|f\|_{L^p_\sigma}.
\end{equation*}
\end{lemma}
\begin{proof}
The conditional sparsity estimate \eqref{eq: sparse property prelim}, the
pairwise disjointness of the sets \(E_j\setminus E_{j+1}\), and
Proposition~\ref{prop: sequential weighted doob} give
\begin{align*}
\sum_{j=0}^\infty
\int_{E_j}
\E_\sigma[f\mid\mathcal F_{\tau_j}]^p\,\mathrm d\mu_\sigma
&\le
\eta^{-1}
\sum_{j=0}^\infty
\int_{E_j\setminus E_{j+1}}
\E_\sigma[f\mid\mathcal F_{\tau_j}]^p\,\mathrm d\mu_\sigma
\le
\eta^{-1}\|M_{\sigma,\mathcal S}f\|_{L^p_\sigma}^p
\le
\eta^{-1}(p')^p\|f\|_{L^p_\sigma}^p.
\end{align*}
This finishes the proof.
\end{proof}

\section[A sequential two weight T1 theorem]{A sequential two weight
\texorpdfstring{$T1$}{T1} theorem}\label{sec:continuous-time-t1}
In this section we prove Theorem~\ref{thm: stoch t1 seq test}, which characterizes the two weight norm of the positive sequential operator \(T\) by direct and dual testing conditions. We treat sufficiency and necessity separately, adapting the arguments of \cite[Theorem~6.1]{HytonenA2remarks} and \cite[Sections~3.A and 3.B]{hanninen_two-weight_2016} to the stopping-time setting.

Throughout this section, we assume the hypotheses and use
the notation of Theorem~\ref{thm: stoch t1 seq test}.

\subsection{Sufficiency}
 In the sufficiency argument, the principal stopping sequences constructed in Lemma \ref{lem:weighted-principal-stopping-sequence} replace the principal cubes used in \cite{hanninen_two-weight_2016, HytonenA2remarks}.

\begin{proof}[Proof of sufficiency in Theorem~\ref{thm: stoch t1 seq test}]
By positivity and duality, it suffices to prove that
\[
\ip{T(f\sigma),g\omega}
\lesssim_{p,q}
(\mathfrak T_{p,q}+\mathfrak T^*_{p,q})
\nrm f_{L^p_\sigma}\nrm g_{L^{q'}_\omega},
\]
for nonnegative $f\in L^p_\sigma$ and $g\in L^{q'}_\omega$.
\proofpart{Principal stopping sequences and decomposition of the pairing}

Apply
Lemma~\ref{lem:weighted-principal-stopping-sequence} to \((f,\sigma)\) and
\((g,\omega)\), denote the resulting principal indices by \(N_j\) and
\(M_i\), respectively, and define the stopping times 
\(\tau_j:=T_{N_j}\) and \(\upsilon_i:=T_{M_i}\).
Then \(\{\tau_j\}_{j=0}^\infty\) is \(\tfrac12\)-sparse with respect to
\(\mu_\sigma\), and \(\{\upsilon_i\}_{i=0}^\infty\) is
\(\tfrac12\)-sparse with respect to \(\mu_\omega\). Set
\(E_j:=\{\tau_j<\infty\}\) and
\(\widetilde E_i:=\{\upsilon_i<\infty\}\).

To rewrite the pairing, observe that
\(\lambda_k\E[f\sigma\mid\mathcal F_{T_k}]\) is
\(\mathcal F_{T_k}\)-measurable. Therefore
\begin{equation}
\begin{aligned}
\ip{T(f\sigma),g\omega}
&=
\sum_{k=0}^\infty
\ip{\lambda_k\E[f\sigma\mid\mathcal F_{T_k}],\E[g\omega\mid\mathcal F_{T_k}]}\\
&=
\sum_{k=0}^\infty
\ips{\lambda_k\E[\sigma\mid\mathcal F_{T_k}]\,\E_\sigma[f\mid\mathcal F_{T_k}],
\E[\omega\mid\mathcal F_{T_k}]\,\E_\omega[g\mid\mathcal F_{T_k}]}. \label{eq:pairing-rewrite}
\end{aligned}\end{equation}

To decompose the pairing, note that if $k<N_0$, then
$\E_\sigma[f\mid\mathcal F_{T_k}]=0$, and if \(k<M_0\), then
$\E_\omega[g\mid\mathcal F_{T_k}]=0$. Hence only
\(k\ge \max\{N_0, M_0\}\) contribute to
\eqref{eq:pairing-rewrite}. For each such \(k\) and any \(x\in X\), there are
unique indices \(i,j\ge0\) such that
\begin{equation*}
N_j(x)\le k<N_{j+1}(x),
\qquad
M_i(x)\le k<M_{i+1}(x).
\end{equation*}
Exactly one of the relations \(N_j(x)\le M_i(x)\) or \(M_i(x)<N_j(x)\) holds. 
\begin{align*}
\mathrm I
&:=
\sum_{i,j,k=0}^\infty
\ips{
\mathbf 1_{\{N_j\le M_i\le k<\min(M_{i+1},N_{j+1})\}}
\lambda_k\E[\sigma\mid\mathcal F_{T_k}]\E_\sigma[f\mid\mathcal F_{T_k}],
\E[g\omega\mid\mathcal F_{T_k}]
},\\
\mathrm{II}
&:=
\sum_{i,j,k=0}^\infty
\ips{
\mathbf 1_{\{M_i<N_j\le k<\min(M_{i+1},N_{j+1})\}}
\lambda_k\E[f\sigma\mid\mathcal F_{T_k}],
\E[\omega\mid\mathcal F_{T_k}]\,\E_\omega[g\mid\mathcal F_{T_k}]
}.
\end{align*}
These two regions partition the nonzero summands in
\eqref{eq:pairing-rewrite}, so
\begin{equation*}
\ip{T(f\sigma),g\omega}=\mathrm I+\mathrm{II}.
\end{equation*}

\proofpart{Estimate of I}
For $i,j,k\geq 0$ define
\begin{align*}
    R_{ijk}&:= \{N_j\le M_i\le k<\min(M_{i+1},N_{j+1})\}\cap \widetilde E_{i}\\
    D_{ij}&:=\{N_j\le M_i<N_{j+1}\}\cap \widetilde E_{i}.
\end{align*}
Then $R_{ijk}$ is \(\mathcal F_{T_k}\)-measurable and $R_{ijk} \subseteq D_{ij}$. Moreover, on $R_{ijk}$ Lemma~\ref{lem:weighted-principal-stopping-sequence} gives the two bounds
\begin{equation*}
\E_\sigma[f\mid\mathcal F_{T_k}]
\le
2\,\E_\sigma[f\mid\mathcal F_{\tau_j}],
\qquad
\E_\omega[g\mid\mathcal F_{T_k}]
\le
2\,\E_\omega[g\mid\mathcal F_{\upsilon_i}].
\end{equation*}
and \(\tau_j,\upsilon_i\le T_k\). After multiplication by $\mb{1}_{R_{ijk}}$,  \eqref{eq:stopped-sigma-algebra-facts} shows that all the factors in I are \(\mathcal F_{T_k}\)-measurable, so we may replace \(\E[\omega\mid\mathcal F_{T_k}]\,\mathrm d\mu\) by \(\mathrm d\mu_\omega\). Using
\eqref{eq: weighted cond exp identity}, positivity, and the fact that
$N_j\le k$ implies $\tau_j \leq T_k$, we obtain
\begin{align*}
\mathrm I
&\le
4\sum_{i,j=0}^\infty\sum_{k=0}^\infty
\int_X
\mathbf 1_{R_{ijk}}
\lambda_k\E[\sigma\mid\mathcal F_{T_k}]
\E_\sigma[f\mid\mathcal F_{\tau_j}]
\E_\omega[g\mid\mathcal F_{\upsilon_i}]\,\mathrm d\mu_\omega\\
&\le
4\sum_{j=0}^\infty
\biggl\langle
\E_\sigma[f\mid\mathcal F_{\tau_j}]\,T_{\tau_j}(\sigma),
\sum_{i=0}^\infty
\E_\omega[g\mid\mathcal F_{\upsilon_i}]\mathbf 1_{D_{ij}}\,\omega
\biggr\rangle\\
&=: 4\sum_{j=0}^\infty
\biggl\langle
\E_\sigma[f\mid\mathcal F_{\tau_j}]\,T_{\tau_j}(\sigma),G_j\,\omega
\biggr\rangle.
\end{align*}
We will decompose $\E_\sigma[f\mid\mathcal F_{\tau_j}]$ using the level sets
\begin{equation*}
A_{j,m}:=E_j\cap
\left\{2^m<\E_\sigma[f\mid\mathcal F_{\tau_j}]\le 2^{m+1}\right\}, \qquad m \in \Z.
\end{equation*}
Each $A_{j,m}$ has finite $\sigma$-measure, as $f \in L^p_\sigma$.
Since
$A_{j,m}\in\mathcal F_{\tau_j}$, Hölder's inequality gives
\begin{align*}
\mathrm I
&\lesssim
\sum_{j,m}
2^m
\bigl\langle
\mathbf 1_{A_{j,m}}T_{\tau_j}(\sigma),
G_j\omega
\bigr\rangle
\le
\sum_{j,m}
2^m
\nrm{\mathbf 1_{A_{j,m}}T_{\tau_j}(\sigma)}_{L^q_\omega}
\nrm{\mathbf 1_{A_{j,m}}G_j}_{L^{q'}_\omega}.
\end{align*}
Since \(\frac1p+\frac1{\theta}+\frac{1}{q'}\ge1\), Hölder's inequality for sequences, together with the monotonicity of sequence norms when \(p\le q\), gives 
\begin{align*}
\mathrm I
&\lesssim
\Bigl(\sum_{j,m}2^{mp}\sigma(A_{j,m})\Bigr)^{1/p}
\biggl\|
\biggl\{
\frac{\nrm{\mathbf 1_{A_{j,m}}T_{\tau_j}(\sigma)}_{L^q_\omega}}
{\sigma(A_{j,m})^{1/p}}
\biggr\}_{j,m}
\biggr\|_{\ell^\theta}
\Bigl(\sum_{j,m}
\nrm{\mathbf 1_{A_{j,m}}G_j}_{L^{q'}_\omega}^{q'}
\Bigr)^{1/q'}.
\end{align*}
We will estimate the three resulting factors by \(\|f\|_{L^p_\sigma}\), \(\mathfrak T_{p,q}\), and \(\|g\|_{L^{q'}_\omega}\), respectively.

For the first factor, since $\{\tau_j\}_{j=0}^\infty$ is $\tfrac12$-sparse with respect to
$\mu_\sigma$, 
Lemma~\ref{lemma: carleson embedding} gives
\begin{equation*}
\Bigl(\sum_{j,m}2^{mp}\sigma(A_{j,m})\Bigr)^{1/p}
\le
\Bigl(
\sum_{j\ge0}\int_{E_j}
\E_\sigma[f\mid\mathcal F_{\tau_j}]^p\,\mathrm d\mu_\sigma
\Bigr)^{1/p}
\lesssim_p
\nrm f_{L^p_\sigma}.
\end{equation*}
For the second factor, 
$\{A_{j,m}\}_{j,m}$ is subordinate to the
$\tfrac12$-sparse sequence $\{\tau_j\}_{j=0}^\infty$ with respect to
\(\mu_\sigma\). Hence, by the definition of \(\mathfrak T_{p,q}\),
\begin{equation*}
\left\|
\left\{
\frac{\nrm{\mathbf 1_{A_{j,m}}T_{\tau_j}(\sigma)}_{L^q_\omega}}
{\sigma(A_{j,m})^{1/p}}
\right\}_{j,m}
\right\|_{\ell^\theta}
\le \mathfrak T_{p,q}.
\end{equation*}
For the third factor, note that 
pairwise disjointness of the sets $A_{j,m}$ in $m$ for fixed $j$ gives
\begin{align*}
\sum_{j,m}\nrm{\mathbf 1_{A_{j,m}}G_j}_{L^{q'}_\omega}^{q'}
&\le
\sum_{j\ge0}\int_X
\Bigl(
\sum_{i=0}^\infty
\E_\omega[g\mid\mathcal F_{\upsilon_i}]\mathbf 1_{D_{ij}}
\Bigr)^{q'}\,\mathrm d\mu_\omega.
\end{align*}
Fix \(j\) and \(x\). The relevant indices are those with
\(N_j(x)\le M_i(x)<N_{j+1}(x)\).
Along these indices the numbers
$\E_\omega[g\mid\mathcal F_{\upsilon_i}](x)$ double, and hence
\begin{equation*}
\Bigl(
\sum_{i=0}^\infty
\E_\omega[g\mid\mathcal F_{\upsilon_i}]\mathbf 1_{D_{ij}}
\Bigr)^{q'}
\lesssim_q
\sum_{i=0}^\infty
\E_\omega[g\mid\mathcal F_{\upsilon_i}]^{q'}\mathbf 1_{D_{ij}}.
\end{equation*}
For fixed $i$, the sets $\{D_{ij}\}_{j\ge0}$ are pairwise disjoint, and
\(D_{ij}\subset\widetilde E_i\). Hence
\begin{align*}
\sum_{j,m}\nrm{\mathbf 1_{A_{j,m}}G_j}_{L^{q'}_\omega}^{q'}
&\lesssim_q
\sum_{i,j\ge0}
\int_{D_{ij}}
\E_\omega[g\mid\mathcal F_{\upsilon_i}]^{q'}\,\mathrm d\mu_\omega
\\
&\le
\sum_{i\ge0}\int_{\widetilde E_i}
\E_\omega[g\mid\mathcal F_{\upsilon_i}]^{q'}\,\mathrm d\mu_\omega
\lesssim_q
\|g\|_{L^{q'}_\omega}^{q'},
\end{align*}
where the final inequality follows from
Lemma~\ref{lemma: carleson embedding}, since
$\{\upsilon_i\}_{i=0}^\infty$ is $\tfrac12$-sparse with respect to $\mu_\omega$.
Combining the three bounds gives
\begin{equation*}
\mathrm I\lesssim_{p,q} \mathfrak T_{p,q}\,\nrm f_{L^p_\sigma}\nrm g_{L^{q'}_\omega}.
\end{equation*}

\proofpart{Estimate of II}
In the proof of the estimate for I, interchange
\[
(f,p,\sigma,N_j,\tau_j)
\quad\text{and}\quad
(g,q',\omega,M_i,\upsilon_i).
\]
The sequence exponent is unchanged because $\frac1{p'}-\frac1{q'}=\frac{1}q-\frac1p$, and the direct testing constant becomes the dual testing constant $\mathfrak T^*_{p,q}$. 
Hence
\begin{equation*}
\mathrm{II}\lesssim_{p,q}
\mathfrak T^*_{p,q}\,\nrm f_{L^p_\sigma}\nrm g_{L^{q'}_\omega}.
\end{equation*}
Combining the estimates for $\mathrm I$ and $\mathrm{II}$ with the pairing
decomposition finishes the proof.
\end{proof}

\subsection{Necessity}
For necessity, we will use that the pointwise localization
\(\mathbf 1_A T_\tau(\sigma)\le T(\mathbf 1_A\sigma)\)
reduces the direct testing condition to a vector-valued estimate for normalized indicators, with sparsity controlling their sum. The dual testing condition follows from the self-adjointness of \(T\).

\begin{proof}[Proof of necessity in Theorem~\ref{thm: stoch t1 seq test}]
Fix a stopping time $\tau$ and a set $A\in\mathcal F_\tau$ with
$0<\sigma(A)<\infty$. For every $k\ge0$, we have
\(A\cap\{T_k\ge\tau\}\in\mathcal F_{T_k}\) by
\eqref{eq:stopped-sigma-algebra-facts}.
Hence
\begin{equation*}
\mathbf 1_A\mathbf 1_{\{T_k\ge \tau\}}\E[\sigma\mid \mathcal F_{T_k}]
=
\E[\mathbf 1_A\mathbf 1_{\{T_k\ge \tau\}}\sigma\mid \mathcal F_{T_k}]
\le
\E[\mathbf 1_A\sigma\mid \mathcal F_{T_k}],
\end{equation*}
and therefore, by positivity,
\begin{equation*}
\mathbf 1_A\,T_\tau(\sigma)
=
\sum_{k=0}^\infty
\lambda_k\,
\mathbf 1_A\mathbf 1_{\{T_k\ge \tau\}}
\E[\sigma\mid \mathcal F_{T_k}]
\le
\sum_{k=0}^\infty
\lambda_k\,\E[\mathbf 1_A\sigma\mid \mathcal F_{T_k}]
=
T(\mathbf 1_A\sigma).
\end{equation*}

\proofpart{Direct testing}
Let $\mathcal P=\{A_{j,m}\}_{j,m}$ be a finite family subordinate to a $\tfrac12$-sparse sequence $\{\tau_j\}_{j=0}^\infty$ with respect to $\mu_\sigma$, where \(0<\sigma(A_{j,m})<\infty\) for every \(j,m\).
Write
\(\phi_{j,m}:=\mathbf 1_{A_{j,m}}/\sigma(A_{j,m})^{1/p}\).
These normalized indicators satisfy
\begin{equation*}
\nrms{\sum_{j,m} \beta_{j,m}\phi_{j,m}}_{L^p_\sigma}
\lesssim_p
\Bigl(\sum_{j,m}\beta_{j,m}^p\Bigr)^{1/p}
\end{equation*}
for nonnegative finitely supported coefficients \(\beta_{j,m}\).
Indeed, by duality, let $h$ be nonnegative with
$\|h\|_{L^{p'}_\sigma}\le1$. Because
\(A_{j,m}\in\mathcal F_{\tau_j}\), the integral of \(h\) over \(A_{j,m}\)
is unchanged if \(h\) is replaced by
\(\E_\sigma[h\mid\mathcal F_{\tau_j}]\). Hölder's inequality and the
disjointness of the regions at each level therefore give
\begin{align*}
\int_X
\sum_{j,m} \beta_{j,m}\phi_{j,m} h\,\mathrm d\mu_\sigma
&\le
\Bigl(\sum_{j,m}\beta_{j,m}^p\Bigr)^{1/p}
\Bigl(
\sum_j
\int_{E_j}
\E_\sigma[h\mid\mathcal F_{\tau_j}]^{p'}\,\mathrm d\mu_\sigma
\Bigr)^{1/p'}
\lesssim_p
\Bigl(\sum_{j,m}\beta_{j,m}^p\Bigr)^{1/p}
\|h\|_{L^{p'}_\sigma}.
\end{align*}
The last inequality follows from Lemma~\ref{lemma: carleson embedding}.

Applying \cite[Proposition~3.1]{hanninen_two-weight_2016} to
$T(\cdot\sigma)$ and the family $\{\phi_{j,m}\}_{j,m}$, and then using the
preceding estimate, gives
\begin{equation*}
\biggl\|
\biggl\{
\nrm{T(\phi_{j,m}\sigma)}_{L^q_\omega}
\biggr\}_{j,m}
\biggr\|_{\ell^\theta}
\lesssim_p
\|T(\cdot\sigma)\|_{L^p_\sigma\to L^q_\omega}.
\end{equation*}
Using $\mathbf 1_{A_{j,m}}T_{\tau_j}(\sigma)\le
T(\mathbf 1_{A_{j,m}}\sigma)$, and then taking the supremum over
$\mathcal P$, gives
\begin{equation*}
\mathfrak T_{p,q}\lesssim_p\|T(\cdot\sigma)\|_{L^p_\sigma\to L^q_\omega}.
\end{equation*}

\proofpart{Dual testing}
Since each $\lambda_k$ is $\mathcal F_{T_k}$-measurable, every summand in
\(T\), and hence \(T\) itself, is self-adjoint with respect to the \(\mu\)-pairing.  
Applying the direct testing argument with
\[
(p,q,\sigma,\omega)
\quad\text{replaced by}\quad
(q',p',\omega,\sigma),
\]
and using $(\frac{1}{p'}-\frac{1}{q'})_+=(\frac1q-\frac1p)_+$, we obtain
\[
\mathfrak T^*_{p,q}
\lesssim_{p,q}
\|T(\cdot\omega)\|_{L^{q'}_\omega\to L^{p'}_\sigma}
=
\|T(\cdot\sigma)\|_{L^p_\sigma\to L^q_\omega}.
\]
Together with the direct testing estimate, this completes the proof.
\end{proof}

When \(p<q\), boundedness of $T$ also restricts where
the coefficients $\lambda_k$ can be non-zero, as we show in the next proposition.
\begin{proposition}
\label{prop:atomicity-p-less-q}
Let \(1<p<q<\infty\), and suppose that
\(T(\cdot\sigma):L^p_\sigma\to L^q_\omega\) is bounded.
Then, for every \(k\ge0\), \(\lambda_k=0\)
\(\mu\)-almost everywhere on the nonatomic part of
\((X,\mathcal F_{T_k},\mu)\).
In particular, if each
\((X,\mathcal F_{T_k},\mu)\) is nonatomic, then \(T=0\).
\end{proposition}

\begin{proof}
Fix \(k\ge0\). By the \(\sigma\)-finiteness of \(\mu_\sigma|_{\mathcal F_{T_k}}\), it suffices to show that \(\lambda_k=0\) $\mu$-almost everywhere on every \(A\in\mathcal F_{T_k}\) with \(0<\sigma(A)<\infty\) contained in the nonatomic part of \((X,\mathcal F_{T_k},\mu)\).

Since \(\sigma>0\), the restriction
\(\mu_\sigma|_{\mathcal F_{T_k}}\) is also nonatomic
on \(A\). For each \(n\ge1\), partition \(A\) into
sets \(A_1,\ldots,A_n\in\mathcal F_{T_k}\) with
\(\sigma(A_i)=\sigma(A)/n\).
Since \(A_i\in\mathcal F_{T_k}\), positivity and
boundedness give
\[
\int_A
\bigl(\lambda_k\E[\sigma\mid\mathcal F_{T_k}]\bigr)^q
\,\mathrm d\mu_\omega
\le
\sum_{i=1}^n
\|T(\mathbf1_{A_i}\sigma)\|_{L^q_\omega}^q
\le
\|T(\cdot\sigma)\|_{L^p_\sigma\to L^q_\omega}^q
\sigma(A)^{q/p}n^{1-q/p}.
\]
Since \(p<q\), the right-hand side tends to zero
as \(n\to\infty\).
The positivity of \(\omega\) and
\(\E[\sigma\mid\mathcal F_{T_k}]\) then gives
\(\lambda_k=0\) $\mu$-almost everywhere on \(A\).
\end{proof}

\section[Testing reduction for continuous-time sparse operators]{Testing
reduction for continuous-time sparse operators}\label{sec:sparse-testing}
In this section we reduce strong-type bounds for continuous-time sparse operators, with \(1<p<\infty\), to conditional testing estimates. When \(r<p\), linearization reduces the bound to the direct and dual testing estimates; when \(r\ge p\), subadditivity and principal stopping reduce it to the direct testing estimate. For \(r<p\), the same linearization also reduces the weak-type bound to the dual testing estimate. These arguments generalize the dyadic arguments in \cite[Sections~2.1 and~2.2]{hytonen_weak_2017}.

Throughout this section, {fix \(1<p<\infty\) and \(r>0\). Let \(\mathcal S=\{\tau_k\}_{k=0}^\infty\) be an increasing, not necessarily sparse, sequence of stopping times}, and let \(\sigma,\omega\) be weights. Write
\[
E_k:=\{\tau_k<\infty\},
\qquad
\sigma_k:=\E[\sigma\mid\mathcal F_{\tau_k}], \qquad k\geq 0.
\]
For a stopping time \(\tau\), set
\begin{align*}
    \mathcal H_\tau&:=\sum_{k=0}^\infty
\sigma_k^r\mathbf 1_{E_k}\mathbf 1_{\{\tau_k\ge\tau\}},\\
\mathcal H^*_\tau
&:=
\sum_{k=0}^\infty
\sigma_k^{r-1}\mathbf 1_{\{\sigma_k>0\}}
\E[\omega\mid\mathcal F_{\tau_k}]
\mathbf 1_{E_k}\mathbf 1_{\{\tau_k\ge\tau\}}.
\end{align*}
Define the direct and dual testing constants by
\begin{align}
\label{eq: mathfrak tr definition}
\mathfrak T_r
&:=
\sup_\tau
\biggl\|
\frac{
\E\bigl[
\mathcal H_\tau^{p/r}
\omega
\mathrel{\big|}\mathcal F_\tau
\bigr]^{r/p}
}{
\E[\sigma\mid\mathcal F_\tau]^{r/p}
}
\biggr\|_{L^\infty_\mu},&&1<p<\infty,\\
\label{eq: mathfrak tr definition 2}
\mathfrak T_r^*
&:=
\sup_\tau
\biggl\|
\frac{
\E\bigl[
 (\mathcal H_\tau^*)^{(p/r)'}\sigma
\mathrel{\big|}\mathcal F_\tau
\bigr]^{1/(p/r)'}
}
{\E[\omega\mid\mathcal F_\tau]^{1/(p/r)'}}
\biggr\|_{L^\infty_\mu},&&r<p<\infty,
\end{align}
where the suprema range over stopping times.

\subsection{Linearization when $p>r$}

In the case $p>r$ we will first use the stopping-time analogue of the linearization in \cite[Lemma~2.1]{hytonen_weak_2017}, with the weighted dyadic maximal operator replaced by the sequential weighted maximal operator from Proposition~\ref{prop: sequential weighted doob}. In \cite{hytonen_weak_2017}, the two weight testing theorem for positive dyadic operators \cite[Proposition~2.3]{hytonen_weak_2017} is used afterwards. In our setting, we use Theorem~\ref{thm: stoch t1 seq test} for the strong-type bound and \cite[Theorem~1.6]{chen_two-weighted_2020} for the weak-type bound.
\begin{proposition}\label{prop:case-p-greater-r}
Assume \(r<p\). Then
\begin{align*}
\|\mathcal A_{\mathcal S}^r(\cdot\sigma)\|
_{L^p_\sigma\to L^p_\omega}^r
&\lesssim_{p,r}\mathfrak T_r+\mathfrak T_r^*,\\
\|\mathcal A_{\mathcal S}^r(\cdot\sigma)\|
_{L^p_\sigma\to L^{p,\infty}_\omega}^r
&\lesssim_{p,r}\mathfrak T_r^*.
\end{align*}
\end{proposition}
\begin{proof}
Set \(s:=p/r>1\) and define the positive linear operator
\[
Lh
:=
\sum_{j=0}^\infty
\sigma_j^{r-1}\mathbf 1_{\{\sigma_j>0\}}
\mathbf 1_{E_j}\E[h\mid\mathcal F_{\tau_j}].
\]
Let \(f\in L^p_\sigma\) be nonnegative and put
\(F:=M_{\sigma,\mathcal S}f\).
Since \(F\ge\E_\sigma[f\mid\mathcal F_{\tau_j}]\) on \(E_j\),
\eqref{eq: weighted cond exp identity} gives
\[
\mathcal A_{\mathcal S}^r(f\sigma)^r
=
\sum_{j=0}^\infty
\sigma_j^r\E_\sigma[f\mid\mathcal F_{\tau_j}]^r
\mathbf 1_{E_j}
\le
\sum_{j=0}^\infty
\sigma_j^r\E_\sigma[F^r\mid\mathcal F_{\tau_j}]
\mathbf 1_{E_j}
=
L(F^r\sigma).
\]
By Proposition~\ref{prop: sequential weighted doob},
\(
\|F^r\|_{L^s_\sigma}
=
\|F\|_{L^p_\sigma}^r
\le
(p')^r\|f\|_{L^p_\sigma}^r.
\)
Taking strong and weak norms, respectively, therefore gives
\begin{align*}
\|\mathcal A_{\mathcal S}^r(\cdot\sigma)\|
_{L^p_\sigma\to L^p_\omega}^r
&\le
(p')^r
\|L(\cdot\sigma)\|_{L^s_\sigma\to L^s_\omega},\\
\|\mathcal A_{\mathcal S}^r(\cdot\sigma)\|
_{L^p_\sigma\to L^{p,\infty}_\omega}^r
&\le
(p')^r
\|L(\cdot\sigma)\|_{L^s_\sigma\to L^{s,\infty}_\omega}.
\end{align*}

For every stopping time \(\tau\), we have
\[
L_\tau(\sigma)=\mathcal H_\tau
\qquad\text{and}\qquad
L_\tau(\omega)=\mathcal H_\tau^*.
\]
By the definitions of \(\mathfrak T_r\) and
\(\mathfrak T_r^*\), for \(A,B\in\mathcal F_\tau\) with
\(0<\sigma(A),\omega(B)<\infty\), we have
\[
\|\mathbf 1_A L_\tau(\sigma)\|_{L^s_\omega}
\le \mathfrak T_r\,\sigma(A)^{1/s}\qquad\text{and}\qquad
\|\mathbf 1_B L_\tau(\omega)\|_{L^{s'}_\sigma}
\le \mathfrak T_r^*\,\omega(B)^{1/s'}.
\]
The coefficients of \(L\) are nonnegative,
\(\mathcal F_{\tau_j}\)-measurable, and vanish on
\(\{\tau_j=\infty\}\).
Applying Theorem~\ref{thm: stoch t1 seq test} with both exponents
equal to \(s\), so that \(\theta=\infty\), therefore yields
\[
\|L(\cdot\sigma)\|_{L^s_\sigma\to L^s_\omega}
\lesssim_s\mathfrak T_r+\mathfrak T_r^*,
\]
proving the strong-type bound.

For the weak-type bound, we use
\cite[Theorem~1.6]{chen_two-weighted_2020}
on the filtration \((\mathcal F_{\tau_j})_{j\ge0}\).
Since \(\tau_j\ge\tau_k\) for \(j\ge k\), we have
\[
\sum_{j\ge k}
\sigma_j^{r-1}\mathbf 1_{\{\sigma_j>0\}}
\E[\omega\mid\mathcal F_{\tau_j}]\mathbf 1_{E_j}
\le \mathcal H_{\tau_k}^*.
\]
The dual estimate with \(\tau=\tau_k\) verifies condition
\cite[(1.11)]{chen_two-weighted_2020}, with both exponents equal to \(s\)
and testing constant at most \(\mathfrak T_r^*\).

By \(\sigma\)-finiteness, choose sets
\(X_n\in\mathcal F_{\tau_0}\), increasing to \(X\), with
\(\mu(X_n)+\sigma(X_n)+\omega(X_n)<\infty\).
Since \(X_n\in\mathcal F_{\tau_j}\) for every \(j\ge0\),
restriction to \(X_n\) preserves the conditional expectations.
We apply \cite[Theorem~1.6]{chen_two-weighted_2020} on each \(X_n\) to finite partial sums
of \(L\) with the coefficients truncated from above.
Neither restriction nor truncation increases the dual
testing constant, so passing to the limit by monotone
convergence gives
\[
\|L(\cdot\sigma)\|_{L^s_\sigma\to L^{s,\infty}_\omega}
\lesssim_s\mathfrak T_r^*.
\]
Together with the linearization estimate, this proves the weak-type bound.
\end{proof}
\subsection{Principal stopping when $p\le r$}
For the strong-type bound when $p\le r$, we follow the principal cube argument of \cite[Section~2.2]{hytonen_weak_2017}, using Lemma~\ref{lem:weighted-principal-stopping-sequence} in place of the usual principal cubes.
\begin{proposition}\label{prop:case-p-less-r}
Assume \(r\ge p\). Then
\begin{equation*}
\|\mathcal A_{\mathcal S}^r(\cdot\sigma)\|_{L^p_\sigma\to L^p_\omega}^r
\lesssim_{p,r}
\mathfrak T_r.
\end{equation*}
\end{proposition}

\begin{proof}
Let \(f\in L^p_\sigma\) be nonnegative.
Apply Lemma~\ref{lem:weighted-principal-stopping-sequence} to
\(f\), \(\sigma\), and the sequence \(\{\tau_j\}_{j=0}^\infty\).
Let \(\{N_m\}_{m=0}^\infty\) be the resulting principal indices and put
\(
\rho_m:=\tau_{N_m}.
\)
Since \(p/r\le1\), subadditivity  gives
\begin{align*}
\mathcal A_{\mathcal S}^r(f\sigma)^p
&\le
\sum_{m=0}^\infty
\Bigl(
\sum_{j=0}^\infty
\E_\sigma[f\mid\mathcal F_{\tau_j}]^r
\sigma_j^r\mathbf 1_{E_j}
\mathbf 1_{\{N_m\le j<N_{m+1}\}}
\Bigr)^{p/r}
\\
&\le
2^p
\sum_{m=0}^\infty
\mathbf 1_{\{\rho_m<\infty\}}
\E_\sigma[f\mid\mathcal F_{\rho_m}]^p
\mathcal H_{\rho_m}^{p/r}.
\end{align*}
By the definition of \(\mathfrak T_r\),
\[
\E[\mathcal H_{\rho_m}^{p/r}\omega
\mid\mathcal F_{\rho_m}]
\le
\mathfrak T_r^{p/r}
\E[\sigma\mid\mathcal F_{\rho_m}].
\]
Since
\(
\mathbf 1_{\{\rho_m<\infty\}}
\E_\sigma[f\mid\mathcal F_{\rho_m}]^p
\)
is \(\mathcal F_{\rho_m}\)-measurable, conditioning each summand gives
\begin{align*}
\|\mathcal A_{\mathcal S}^r(f\sigma)\|_{L^p_\omega}^p
&\le
2^p
\sum_{m=0}^\infty
\E\Bigl[
\mathbf 1_{\{\rho_m<\infty\}}
\E_\sigma[f\mid\mathcal F_{\rho_m}]^p
\mathcal H_{\rho_m}^{p/r}\omega
\Bigr]
\\
&\le
2^p\mathfrak T_r^{p/r}
\sum_{m=0}^\infty
\int_{\{\rho_m<\infty\}}
\E_\sigma[f\mid\mathcal F_{\rho_m}]^p
\,\mathrm d\mu_\sigma
\lesssim_p
\mathfrak T_r^{p/r}\|f\|_{L^p_\sigma}^p,
\end{align*}
where the last estimate follows from Lemma~\ref{lemma: carleson embedding} applied to \(\{\rho_m\}_{m=0}^\infty\), which is \(\tfrac12\)-sparse with respect to \(\mu_\sigma\).
\end{proof}

\section{Continuous-time testing estimates}
\label{sec:testing-to-characteristics}

In this section we estimate the direct and dual testing constants in terms of the \(A_p\) and \(A_\infty\) characteristics. We follow the dyadic argument of \cite[Section~3]{hytonen_weak_2017}, using a conditional principal lemma and sparse packing estimates for general filtrations. 

Together with Propositions~\ref{prop:case-p-greater-r} and \ref{prop:case-p-less-r},
the following estimates imply Theorem~\ref{thm: sparse operators ap ainfty bound} and the case \(r<p\) of Theorem~\ref{thm: sparse operators weak type bound}.

\begin{proposition}\label{prop: testing conditions bounded by characteristics}
Let $1<p<\infty$, let $r>0$, and let
$\mathcal S=\{\tau_j\}_{j=0}^\infty$ be an $\eta$-sparse sequence of stopping
times. Let $\sigma,\omega$ be weights, and let $\mathfrak T_r$ and, when
$r<p$, $\mathfrak T_r^*$ be defined by
\eqref{eq: mathfrak tr definition} and
\eqref{eq: mathfrak tr definition 2}. Then
\begin{align*}
\mathfrak T_r
&\lesssim_{p,r,\eta}
[\omega,\sigma]_{A_p}^{r/p}[\sigma]_{A_\infty}^{r/p},
\\
\mathfrak T_r^*
&\lesssim_{p,r,\eta}
[\omega,\sigma]_{A_p}^{r/p}[\omega]_{A_\infty}^{1-r/p},
\qquad r<p.
\end{align*}
\end{proposition}

{For the testing estimates below, fix \(p,r,\mathcal S,\sigma,\omega\) as in Proposition~\ref{prop: testing conditions bounded by characteristics} and write}
\[
\sigma_k:=\E[\sigma\mid\mathcal F_{\tau_k}],
\qquad
\omega_k:=\E[\omega\mid\mathcal F_{\tau_k}],
\qquad
E_k:=\{\tau_k<\infty\}.
\]
\subsection{Conditional Principal Lemma}\label{sec:principal-lemma}

The testing constants involve powers of sums over the sparse sequence. To handle them, we use a conditional version of the principal lemma of Cascante--Ortega--Verbitsky \cite[Proposition~2.2]{COV2004}. We derive it from the martingale version of Tanaka and Terasawa \cite[Lemma~2.3]{tanaka_positive_2013}. Throughout this subsection, conditional expectations are taken with respect to \(\nu\).

\begin{lemma}[{\cite[Lemma~2.3]{tanaka_positive_2013}}]
\label{lemma: principal lemma expectation version}
Let \((\Omega,\nu,\mathcal F,(\mathcal F_k)_{k\in\mathbb Z})\) be a filtered measure space. Assume that \(\nu|_{\mathcal F_k}\) is \(\sigma\)-finite for every \(k\). 
Let \(\{\alpha_k\}_{k\in\mathbb Z}\) be bounded nonnegative \(\mathcal F_k\)-measurable functions, and set \(\widetilde\alpha_k:=\sum_{\ell\ge k}\alpha_\ell\). 
If \(s>1\) and \(w\ge0\) is measurable and both integrals below are finite, then
\begin{equation*}
\int_\Omega
\Bigl(\sum_{k\in\mathbb Z}\alpha_k\E[w\mid\mathcal F_k]\Bigr)^s
\,\mathrm d\nu
\simeq_s
\int_\Omega\sum_{k\in\mathbb Z}
\alpha_k\E[w\mid\mathcal F_k]
\bigl(\E[\widetilde\alpha_k w\mid\mathcal F_k]\bigr)^{s-1}
\,\mathrm d\nu.
\end{equation*}
\end{lemma}

We need the same equivalence after conditioning on an earlier \(\sigma\)-algebra.

\begin{lemma}
\label{lemma: principal lemma conditional version}
Fix \(j\in\mathbb Z\). Let \((\Omega,\nu,\mathcal F)\) be a measure space
equipped with a filtration \((\mathcal F_k)_{k\ge j}\), and assume that
\(\nu|_{\mathcal F_k}\) is \(\sigma\)-finite for every \(k\ge j\). Let
\(\{\alpha_k\}_{k\ge j}\) be nonnegative
\(\mathcal F_k\)-measurable functions, and set
\(\widetilde\alpha_k:=\sum_{\ell\ge k}\alpha_\ell\). Then, for every \(s>1\) and
every measurable \(w\ge0\),
\begin{equation*}
\E\Bigl[
\Bigl(\sum_{k\ge j}\alpha_k\E[w\mid\mathcal F_k]\Bigr)^s
\mathrel{\Big|}\mathcal F_j
\Bigr]
\simeq_s
\E\Bigl[
\sum_{k\ge j}\alpha_k\E[w\mid\mathcal F_k]
\bigl(\E[\widetilde\alpha_k w\mid\mathcal F_k]\bigr)^{s-1}
\mathrel{\Big|}\mathcal F_j
\Bigr]
\end{equation*}
\(\nu\)-almost everywhere.
\end{lemma}
\begin{proof}
By truncation and conditional monotone convergence, it suffices to assume that \(w\) and the sequence \(\{\alpha_k\}_{k\ge j}\) are bounded and that only finitely many \(\alpha_k\) are nonzero. Extend the filtration by setting \(\mathcal F_k=\mathcal F_j\) and \(\alpha_k:=0\) for \(k<j\).

Let \(A\in\mathcal F_j\) with \(\nu(A)<\infty\). Since \(A\in\mathcal F_j\subseteq\mathcal F_k\), for \(k\ge j\)

\[
\E[w\mathbf1_A\mid\mathcal F_k]
=
\mathbf1_A\E[w\mid\mathcal F_k],
\qquad
\E[\widetilde\alpha_k w\mathbf1_A\mid\mathcal F_k]
=
\mathbf1_A\E[\widetilde\alpha_k w\mid\mathcal F_k].
\]
Applying Lemma~\ref{lemma: principal lemma expectation version} to \(w\mathbf1_A\) gives the desired equivalence after integration over \(A\).
Since \(\nu|_{\mathcal F_j}\) is \(\sigma\)-finite, this proves the conditional equivalence.
\end{proof}
\subsection{Sparse packing estimates and stopping-time reduction}

We first prove two packing estimates: one for products whose exponents sum to less than one, and one for a linear sum controlled by the \(A_\infty\) characteristic. These are continuous-time counterparts of the dyadic packing estimates in \cite[Lemma~4.2]{fackler_off-diagonal_2018}, see also \cite[Lemma~5.2]{HytonenA2remarks}.
\begin{lemma}
\label{lemma: two weight sparse power packing}
Let \(\{\tau_k\}_{k=0}^\infty\) be an \(\eta\)-sparse sequence of stopping
times, and set \(E_k:=\{\tau_k<\infty\}\). Let \(u,v\ge0\) be measurable,
let \(\beta,\gamma\ge0\) satisfy \(\beta+\gamma<1\), and let \(j\ge 0\). Then
\begin{equation}
\label{eq: two weight sparse power packing}
\E\Bigl[\sum_{\ell\ge j}
\E[u\mid\mathcal F_{\tau_\ell}]^\beta
\E[v\mid\mathcal F_{\tau_\ell}]^\gamma
\mathbf 1_{E_\ell}\mid\mathcal F_{\tau_j}\Bigr]
\lesssim_{\eta,\beta,\gamma}
\E[u\mid\mathcal F_{\tau_j}]^\beta
\E[v\mid\mathcal F_{\tau_j}]^\gamma.
\end{equation}
\(\mu\)-almost everywhere. Moreover, if \(u\) is a weight, then
\begin{equation}
\label{eq: linear sparse sum controlled by Ainfty}
\E\Bigl[\sum_{k\ge j}\E[u\mid\mathcal F_{\tau_k}]\,\mathbf 1_{E_k}\mid\mathcal F_{\tau_j}\Bigr]
\le
\eta^{-1}[u]_{A_\infty}\E[u\mid\mathcal F_{\tau_j}]
\end{equation}
\(\mu\)-almost everywhere.
\end{lemma}

\begin{proof}
Set \(u_\ell:=\E[u\mid\mathcal F_{\tau_\ell}]\) and \(v_\ell:=\E[v\mid\mathcal F_{\tau_\ell}]\). 
For \(\ell\ge j\), conditional H\"older's inequality, the tower property, and iteration of \eqref{eq: sparse property prelim} give
\[
\E[u_\ell^\beta v_\ell^\gamma\mathbf 1_{E_\ell}
\mid\mathcal F_{\tau_j}]
\le
u_j^\beta v_j^\gamma
\E[\mathbf 1_{E_\ell}\mid\mathcal F_{\tau_j}]^{1-\beta-\gamma}
\le
u_j^\beta v_j^\gamma
(1-\eta)^{(\ell-j)(1-\beta-\gamma)}
\mathbf 1_{E_j}.
\]
Summing the geometric series proves \eqref{eq: two weight sparse power packing}.

For \eqref{eq: linear sparse sum controlled by Ainfty}, \eqref{eq: sparse property prelim} and the tower property give
\[
\E\Bigl[
\sum_{\ell\ge j}u_\ell\mathbf 1_{E_\ell}
\mid\mathcal F_{\tau_j}
\Bigr]
\le
\eta^{-1}
\E\Bigl[
\sum_{\ell\ge j}
u_\ell\mathbf 1_{E_\ell\setminus E_{\ell+1}}
\mid\mathcal F_{\tau_j}
\Bigr]
\le
\eta^{-1}
\E[M_{\tau_j}u\mid\mathcal F_{\tau_j}]
\le
\eta^{-1}[u]_{A_\infty}u_j.
\]
The middle inequality follows because \(u_\ell\le M_{\tau_j}u\) on \(E_\ell\setminus E_{\ell+1}\), and these sets are pairwise disjoint.
\end{proof}
The following estimate will be used in both the direct and dual arguments.
For \(0<r<p\), set
\[
\lambda
:=
\min\Bigl\{
1,\frac{r}{p-1},\frac{r}{p-r}
\Bigr\}.
\]
Then
\[
\frac rp<\lambda\le1,
\qquad
r-\lambda(p-1)\ge0,
\qquad
r-\lambda(p-1)+1-\lambda
=
r+1-p\lambda<1.
\]
On \(E_\ell\), the \(A_p\) condition gives
\[
\sigma_\ell^r\omega_\ell
=
\bigl(\sigma_\ell^{p-1}\omega_\ell\bigr)^\lambda
\sigma_\ell^{\,r-\lambda(p-1)}
\omega_\ell^{\,1-\lambda}
\le
[\omega,\sigma]_{A_p}^{\lambda}
\sigma_\ell^{\,r-\lambda(p-1)}
\omega_\ell^{\,1-\lambda}.
\]
Hence, by Lemma~\ref{lemma: two weight sparse power packing}, for every \(k\ge0\),
\begin{equation}
\label{eq: common conditional tail estimate}
\E\Bigl[
\sum_{\ell\ge k}
\sigma_\ell^r\omega_\ell\mathbf 1_{E_\ell}
\mathrel{\Big|}\mathcal F_{\tau_k}
\Bigr]
\lesssim_{\eta,p,r}
[\omega,\sigma]_{A_p}^{\lambda}
\sigma_k^{\,r-\lambda(p-1)}
\omega_k^{\,1-\lambda}.
\end{equation}

The following lemma passes estimates at the times \(\tau_j\) to
arbitrary stopping times.
\begin{lemma}
\label{lem: reduction to sparse times}
Let \(\{\tau_j\}_{j=0}^\infty\) be an increasing sequence of stopping times, and let \(a_j\ge0\) be \(\mathcal F_{\tau_j}\)-measurable and vanish on \(\{\tau_j=\infty\}\). 
Let \(q>0\) and \(u,v\ge0\) be measurable. 
Suppose that, for some \(C>0\) and every \(j\ge0\),
\begin{equation}\label{eq: reduction sparse times hypothesis}
\E\Bigl[
\Bigl(\sum_{k\ge j}a_k\Bigr)^qv
\mathrel{\Big|}\mathcal F_{\tau_j}
\Bigr]
\le
C\,\E[u\mid\mathcal F_{\tau_j}].
\end{equation}
Then, for every stopping time \(\tau\),
\[
\E\Bigl[
\Bigl(
\sum_{k=0}^\infty
a_k\mathbf 1_{\{\tau_k\ge\tau\}}
\Bigr)^qv
\mathrel{\Big|}\mathcal F_\tau
\Bigr]
\le
C\,\E[u\mid\mathcal F_\tau].
\]
\end{lemma}
\begin{proof}
Fix a stopping time \(\tau\), and set
\[
A_0:=\{\tau\le\tau_0\},
\qquad
A_j:=\{\tau_{j-1}<\tau\le\tau_j\},
\quad j\ge1.
\]
By \eqref{eq:stopped-sigma-algebra-facts}, the sets \(A_j\) are pairwise disjoint and belong to \(\mathcal F_{\tau_j}\).
Moreover,
\[
\sum_{k=0}^\infty a_k\mathbf 1_{\{\tau_k\ge\tau\}}
=
\sum_{k\ge j}a_k
\quad\text{on }A_j,
\]
and the sum on the left vanishes outside \(\bigcup_{j\ge0}A_j\).

If \(B\in\mathcal F_\tau\), then \(B\cap A_j\in\mathcal F_{\tau_j}\) by \eqref{eq:stopped-sigma-algebra-facts}. Hence, by
\eqref{eq: reduction sparse times hypothesis},
\begin{align*}
\int_B
\Bigl(&
\sum_{k=0}^\infty
a_k\mathbf 1_{\{\tau_k\ge\tau\}}
\Bigr)^qv\,\mathrm d\mu
=
\sum_{j=0}^\infty
\int_{B\cap A_j}
\Bigl(\sum_{k\ge j}a_k\Bigr)^qv\,\mathrm d\mu
\\
&\le
C\sum_{j=0}^\infty
\int_{B\cap A_j}
\E[u\mid\mathcal F_{\tau_j}]\,\mathrm d\mu
=
C\sum_{j=0}^\infty\int_{B\cap A_j}u\,\mathrm d\mu
\le
C\int_Bu\,\mathrm d\mu.
\end{align*}
Since \(\mu|_{\mathcal F_\tau}\) is \(\sigma\)-finite, the conditional inequality follows.
\end{proof}

\subsection{Direct testing estimate}
For \(r<p\), Lemma~\ref{lemma: principal lemma conditional version} and \eqref{eq: common conditional tail estimate} reduce the direct testing estimate to the linear packing estimate for \(\sigma\). We first reduce to this range by monotonicity. 
If \(r\ge p\), then, for every stopping time \(\tau\),
\[
\Bigl(
\sum_{k=0}^\infty
\sigma_k^r\mathbf 1_{E_k}\mathbf 1_{\{\tau_k\ge\tau\}}
\Bigr)^{1/r}
\le
\Bigl(
\sum_{k=0}^\infty
\sigma_k^{p-1}\mathbf 1_{E_k}\mathbf 1_{\{\tau_k\ge\tau\}}
\Bigr)^{1/(p-1)}
\]
by monotonicity in the exponent. Hence
\[
\mathfrak T_r
\le
\mathfrak T_{p-1}^{\,r/(p-1)}.
\]
Therefore, it is enough to consider \(r<p\).

For \(k\ge0\), set
\(
S_k:=\sum_{\ell\ge k}\sigma_\ell^r\mathbf 1_{E_\ell}
\) and
fix \(j\ge0\). Applying
Lemma~\ref{lemma: principal lemma conditional version} to
\((X,\mu_\omega)\), with exponent \(p/r\), \(w=1\), and
\(\alpha_k=\sigma_k^r\mathbf 1_{E_k}\), gives
\begin{equation}\label{eq: principal lemma for T_r}
\E\bigl[S_j^{p/r}\omega\mid\mathcal F_{\tau_j}\bigr]
\lesssim_{p,r}
\E\Bigl[
\sum_{k\ge j}
\sigma_k^r\mathbf 1_{E_k}
\bigl(
\E_\omega[S_k\mid\mathcal F_{\tau_k}]
\bigr)^{\frac{p-r}{r}}
\omega
\mathrel{\Big|}\mathcal F_{\tau_j}
\Bigr].
\end{equation}

By the tower property, \eqref{eq: weighted cond exp identity}, and \eqref{eq: common conditional tail estimate},
\[
\E_\omega[S_k\mid\mathcal F_{\tau_k}]
\lesssim_{\eta,p,r}
[\omega,\sigma]_{A_p}^{\lambda}
\mathbf 1_{\{\omega_k>0\}}
\sigma_k^{\,r-\lambda(p-1)}
\omega_k^{-\lambda}.
\]
Substituting this estimate into \eqref{eq: principal lemma for T_r} and using the tower property, we obtain
\[
\begin{aligned}
\E[S_j^{p/r}\omega\mid\mathcal F_{\tau_j}]
&\lesssim_{\eta,p,r}
[\omega,\sigma]_{A_p}^{\lambda\frac{p-r}{r}}
\E\Bigl[
\sum_{k\ge j}
\mathbf 1_{\{\omega_k>0\}}
\sigma_k
\bigl(\sigma_k^{p-1}\omega_k\bigr)^{
1-\lambda\frac{p-r}{r}}
\mathbf 1_{E_k}
\mathrel{\Big|}\mathcal F_{\tau_j}
\Bigr]
\\
&\lesssim_{\eta,p,r}
[\omega,\sigma]_{A_p}
\E\Bigl[
\sum_{k\ge j}\sigma_k\mathbf1_{E_k}
\mathrel{\Big|}\mathcal F_{\tau_j}
\Bigr]
\lesssim_{\eta,p,r}
[\omega,\sigma]_{A_p}
[\sigma]_{A_\infty}\sigma_j.
\end{aligned}
\]
The second estimate uses the \(A_p\) condition and \(1-\lambda\frac{p-r}{r}\ge0\); the last follows from \eqref{eq: linear sparse sum controlled by Ainfty}.

Applying Lemma~\ref{lem: reduction to sparse times} with
\[
a_k=\sigma_k^r\mathbf 1_{E_k},
\qquad
q=\frac pr,
\qquad
u=\sigma,
\qquad
v=\omega
\]
we obtain, for every stopping time \(\tau\),
\[
\E[\mathcal H_\tau^{p/r}\omega\mid\mathcal F_\tau]
\lesssim_{\eta,p,r}
[\omega,\sigma]_{A_p}
[\sigma]_{A_\infty}
\E[\sigma\mid\mathcal F_\tau].
\]
Taking the \(r/p\)-power in the definition of \(\mathfrak T_r\) yields
\[
\mathfrak T_r
\lesssim_{\eta,p,r}
[\omega,\sigma]_{A_p}^{r/p}
[\sigma]_{A_\infty}^{r/p}.
\]
\subsection{Dual testing estimate}
For the dual test estimate, we apply Lemma~\ref{lemma: principal lemma conditional version} with respect to \(\mu_\sigma\). The same conditional tail estimate \eqref{eq: common conditional tail estimate} applies, and the final packing estimate now involves \(\omega\). 
Assume that \(p>r\). For \(k\ge0\), define
\[
S_k^*
:=
\sum_{\ell\ge k}
\sigma_\ell^{r-1}\mathbf 1_{\{\sigma_\ell>0\}}
\omega_\ell\mathbf 1_{E_\ell}.
\]
Fix an integer \(j\ge0\). Applying
Lemma~\ref{lemma: principal lemma conditional version} to
\((X,\mu_\sigma)\), with exponent \((p/r)'\), \(w=1\), and
\(
\alpha_k
=
\sigma_k^{r-1}\mathbf 1_{\{\sigma_k>0\}}
\omega_k\mathbf 1_{E_k},
\)
gives
\[
\E\bigl[(S_j^*)^{(p/r)'}\sigma
\mid\mathcal F_{\tau_j}\bigr]
\lesssim_{p,r}
\E\Bigl[
\sum_{k\ge j}
\sigma_k^{r-1}\mathbf 1_{\{\sigma_k>0\}}
\omega_k\mathbf 1_{E_k}
\bigl(
\E_\sigma[S_k^*\mid\mathcal F_{\tau_k}]
\bigr)^{\frac{r}{p-r}}
\sigma
\mathrel{\Big|}\mathcal F_{\tau_j}
\Bigr].
\]
By the tower property,
\[
\E[S_k^*\sigma\mid\mathcal F_{\tau_k}]
=
\E\Bigl[
\sum_{\ell\ge k}
\sigma_\ell^r\omega_\ell\mathbf1_{E_\ell}
\mathrel{\Big|}\mathcal F_{\tau_k}
\Bigr]
=
\E[S_k\omega\mid\mathcal F_{\tau_k}].
\]
Hence, by \eqref{eq: weighted cond exp identity} and \eqref{eq: common conditional tail estimate},
\[
\E_\sigma[S_k^*\mid\mathcal F_{\tau_k}]
\lesssim_{\eta,p,r}
[\omega,\sigma]_{A_p}^{\lambda}
\mathbf1_{\{\sigma_k>0\}}
\sigma_k^{\,r-\lambda(p-1)-1}
\omega_k^{\,1-\lambda}.
\]
Substituting this into the conditional principal-lemma bound gives
\[
\begin{aligned}
\E\bigl[(S_j^*)^{(p/r)'}\sigma
\mid\mathcal F_{\tau_j}\bigr]
&\lesssim_{\eta,p,r}
[\omega,\sigma]_{A_p}^{\lambda\frac r{p-r}}
\E\Bigl[
\sum_{k\ge j}
\mathbf1_{\{\sigma_k>0\}}
\omega_k
\bigl(\sigma_k^{p-1}\omega_k\bigr)^{
(1-\lambda)\frac r{p-r}}
\mathbf1_{E_k}
\mathrel{\Big|}\mathcal F_{\tau_j}
\Bigr]
\\
&\lesssim_{\eta,p,r}
[\omega,\sigma]_{A_p}^{\frac r{p-r}}
\E\Bigl[
\sum_{k\ge j}\omega_k\mathbf1_{E_k}
\mathrel{\Big|}\mathcal F_{\tau_j}
\Bigr]
\lesssim_{\eta,p,r}
[\omega,\sigma]_{A_p}^{\frac r{p-r}}
[\omega]_{A_\infty}\omega_j.
\end{aligned}
\]
The second estimate uses the \(A_p\) condition and \((1-\lambda)\frac r{p-r}\ge0\); the last follows from \eqref{eq: linear sparse sum controlled by Ainfty}.

Applying Lemma~\ref{lem: reduction to sparse times} with
\[
a_k
=
\sigma_k^{r-1}\mathbf 1_{\{\sigma_k>0\}}
\omega_k\mathbf 1_{E_k},
\qquad
q=\left(\frac pr\right)',
\qquad
u=\omega,
\qquad
v=\sigma,
\]
we obtain, for every stopping time \(\tau\),
\[
\E[(\mathcal H_\tau^*)^{(p/r)'}\sigma
\mid\mathcal F_\tau]
\lesssim_{\eta,p,r}
[\omega,\sigma]_{A_p}^{\frac r{p-r}}
[\omega]_{A_\infty}
\E[\omega\mid\mathcal F_\tau].
\]
Taking the \(1/(p/r)'=(p-r)/p\)-power in the definition of \(\mathfrak T_r^*\) yields
\[
\mathfrak T_r^*
\lesssim_{\eta,p,r}
[\omega,\sigma]_{A_p}^{r/p}
[\omega]_{A_\infty}^{1-r/p}.
\]
This completes the proof of Proposition~\ref{prop: testing conditions bounded by characteristics}. Combining it with {Propositions~\ref{prop:case-p-greater-r} and \ref{prop:case-p-less-r}} proves Theorem~\ref{thm: sparse operators ap ainfty bound}.
\section{Weak-type estimates}
\label{sec:weak-type}

In this section we prove
Theorem~\ref{thm: sparse operators weak type bound}. For \(1<p<\infty\) and \(r<p\), Theorem~\ref{thm: sparse operators weak type bound} follows directly from Propositions~\ref{prop:case-p-greater-r}
and~\ref{prop: testing conditions bounded by characteristics}. However, we prefer to give a self-contained treatment of all cases using analogous arguments. In our proof, we will group the values of the stopped conditional averages in $\mathcal A_{\mathcal S}^r$. For \(r>p\), these groups can be summed directly. For \(r\le p\), we will also need a weighted counting estimate for the number of averages in each group.

Throughout this section, let \(0<\eta<1\), let
\(\mathcal S=\{\tau_j\}_{j\ge0}\) be an
\(\eta\)-sparse sequence of stopping times, and set
\(E_j:=\{\tau_j<\infty\}\).
For weights \(\sigma,\omega\), write
\[
\sigma_j:=\E[\sigma\mid\mathcal F_{\tau_j}],
\qquad
\omega_j:=\E[\omega\mid\mathcal F_{\tau_j}].
\]

\subsection{Maximal and layer estimates}
\label{sec:maximal-layer-estimates}
We adapt to our setting the argument of Hyt\"onen and Li \cite[Section~4]{hytonen_weak_2017},
which builds on the square-function argument of
Lacey and Scurry \cite{lacey_scurry_2012}.
We begin with weak-type bounds for \(M_{1,\mathcal S}\).
Recall that
\[
M_{1,\mathcal S}h
:=\sup_{j\ge0}\E[|h|\mid\mathcal F_{\tau_j}]\mathbf1_{E_j}
\qquad\text{and}\qquad
M_{\omega,\mathcal S}h
:=\sup_{j\ge0}\E_\omega[|h|\mid\mathcal F_{\tau_j}]
\mathbf1_{E_j}.
\]
\begin{lemma}
\label{lem:weak-sequential-maximal}
~
\begin{enumerate}[label=\normalfont(\roman*)]
\item\label{item:weak-sequential-maximal-p-general}
Let \(1<p<\infty\), and let \(\sigma,\omega\) be weights.
For every \(f\in L^p_\sigma\),\begin{equation*}
\|M_{1,\mathcal S}(f\sigma)\|_{L^{p,\infty}_\omega}
\le [\omega,\sigma]_{A_p}^{1/p}\|f\|_{L^p_\sigma}.
\end{equation*}
\item\label{item:weak-sequential-maximal-p-1}
Let \(v,\omega\) be weights.
For every \(f\in L^1_v\),
\[
\|M_{1,\mathcal S}f\|_{L^{1,\infty}_\omega}
\le [\omega,v]_{A_1}\|f\|_{L^1_v}.
\]
\end{enumerate}
\end{lemma}
\begin{proof}
It suffices to consider \(f\ge0\) and assume that the
relevant weight characteristic is finite.
For part~\ref{item:weak-sequential-maximal-p-general},
conditional H\"older's inequality, the weighted conditional
expectation identity, and the \(A_p\) condition give, on \(E_j\),
\[
\E[f\sigma\mid\mathcal F_{\tau_j}]^p
\le
\sigma_j^{p-1}\E[f^p\sigma\mid\mathcal F_{\tau_j}]=
\sigma_j^{p-1}\omega_j
\E_\omega\Bigl[\frac{f^p\sigma}{\omega}
\Bigm|\mathcal F_{\tau_j}\Bigr]
\le
[\omega,\sigma]_{A_p}
\E_\omega\Bigl[\frac{f^p\sigma}{\omega}
\Bigm|\mathcal F_{\tau_j}\Bigr].
\]
Taking the supremum over \(j\), we obtain
\[
\bigl(M_{1,\mathcal S}(f\sigma)\bigr)^p
\le [\omega,\sigma]_{A_p}\,
M_{\omega,\mathcal S}\Bigl(\frac{f^p\sigma}{\omega}\Bigr).
\]

For part~\ref{item:weak-sequential-maximal-p-1}, the \(A_1\) condition gives
\(\omega_j\le[\omega,v]_{A_1}v\) on \(E_j\). Multiplying by \(f\),
conditioning on \(\mathcal F_{\tau_j}\), and using
\(E_j\in\mathcal F_{\tau_j}\) before taking the supremum gives
\[
M_{1,\mathcal S}f
\le [\omega,v]_{A_1}\,
M_{\omega,\mathcal S}\Bigl(\frac{fv}{\omega}\Bigr).
\]
Both claims now follow from Doob's weak-\((1,1)\) inequality
under \(\mu_\omega\)
\cite[Theorem~3.1]{riederer_refined_2019}.
\end{proof}
We next estimate contributions from stopped averages lying in a fixed geometric range, which we call a coefficient layer.
Following \cite[Section~4]{DomingoSalazarLaceyRey2016} and \cite[Section~4]{hytonen_weak_2017}, the following lemma gives a bound uniform in the layer index.
We define the layers using consecutive powers of \(1-\eta/2\).
This ratio exceeds \(1-\eta\), allowing us to retain a fixed fraction of each stopped average on pairwise disjoint sets.
\begin{lemma}
\label{lem:weak-coefficient-layer}
Let \(m\ge0\) be an integer.
\begin{enumerate}[label=\normalfont(\roman*)]
\item \label{it:weakcoeff1}
Let \(1<p<\infty\), and let \(\sigma,\omega\) be weights.
For a nonnegative \(f\in L^p_\sigma\), define
\begin{equation*}
D_j^m
:=E_j\cap\Bigl\{
\bigl(1-\tfrac\eta2\bigr)^{m+1}
<\E[f\sigma\mid\mathcal F_{\tau_j}]
\le\bigl(1-\tfrac\eta2\bigr)^m
\Bigr\},
\qquad j\ge0.
\end{equation*}
Then
\begin{equation*}
\sum_{j\ge0}
\int_{D_j^m}
\E[f\sigma\mid\mathcal F_{\tau_j}]^p\,\mathrm d\mu_\omega
\le
\Bigl(\frac{2-\eta}{\eta}\Bigr)^p
[\omega,\sigma]_{A_p}\|f\|_{L^p_\sigma}^p.
\end{equation*}
\item \label{it:weakcoeff2}
Let \(v,\omega\) be  weights.
For a nonnegative \(f\in L^1_v\), define
\begin{equation*}
D_j^m
:=E_j\cap\Bigl\{
\bigl(1-\tfrac\eta2\bigr)^{m+1}
<\E[f\mid\mathcal F_{\tau_j}]
\le\bigl(1-\tfrac\eta2\bigr)^m
\Bigr\},
\qquad j\ge0.
\end{equation*}
Then
\begin{equation*}
\sum_{j\ge0}
\int_{D_j^m}
\E[f\mid\mathcal F_{\tau_j}]
\,\mathrm d\mu_\omega
\le
\frac{2-\eta}{\eta}
[\omega,v]_{A_1}\|f\|_{L^1_v}.
\end{equation*}
\end{enumerate}
\end{lemma}

\begin{proof}
We first prove part~\ref{it:weakcoeff1}. 
Set
\[
F_j:=D_j^m\setminus\bigcup_{k>j}D_k^m.
\]
The sets \(F_j\) are pairwise disjoint, but need not
belong to \(\mathcal F_{\tau_j}\).
For fixed \(j\), partition \(D_j^m\setminus F_j\) according to the first index \(k>j\) for which the point belongs to \(D_k^m\). Each piece belongs to \(\mathcal F_{\tau_k}\) and is contained in \(D_k^m\). Moreover, \(D_j^m\setminus F_j\subseteq D_j^m\cap E_{j+1}\). The tower property, conditional monotone convergence, and sparseness therefore give
\[
\E[f\sigma\mathbf1_{D_j^m\setminus F_j}\mid\mathcal F_{\tau_j}]
\le
\left(1-\frac\eta2\right)^m
\E[\mathbf1_{D_j^m\setminus F_j}\mid\mathcal F_{\tau_j}]
\le
(1-\eta)\left(1-\frac\eta2\right)^m\mathbf1_{D_j^m}.
\]
Since \(D_j^m\in\mathcal F_{\tau_j}\), subtracting this bound and using the definition of \(D_j^m\) gives, on \(D_j^m\)
\[
\E[f\sigma\mathbf1_{F_j}\mid\mathcal F_{\tau_j}]
\ge
\Bigl(1-\frac{1-\eta}{1-\eta/2}\Bigr)
\E[f\sigma\mid\mathcal F_{\tau_j}]
=\frac{\eta}{2-\eta}
\E[f\sigma\mid\mathcal F_{\tau_j}].
\]
Conditional H\"older's inequality and the \(A_p\) condition give
\[
\begin{aligned}
\int_{D_j^m}
\E[f\sigma\mid\mathcal F_{\tau_j}]^p\,\mathrm d\mu_\omega
&\le
\Bigl(\frac{2-\eta}{\eta}\Bigr)^p
\int_{D_j^m}
\E[f^p\sigma\mathbf1_{F_j}\mid\mathcal F_{\tau_j}]
\sigma_j^{p-1}\omega_j\,\mathrm d\mu
\\&\le
\Bigl(\frac{2-\eta}{\eta}\Bigr)^p
[\omega,\sigma]_{A_p}
\int_{F_j}f^p\,\mathrm d\mu_\sigma.
\end{aligned}
\]
Summing over the pairwise disjoint sets \(F_j\)
proves part~\ref{it:weakcoeff1}.

For part~\ref{it:weakcoeff2}, repeat the argument with \(f\) in place
of \(f\sigma\) and use the \(A_1\) condition to obtain
\[
\int_{D_j^m}
\E[f\mid\mathcal F_{\tau_j}]\,\mathrm d\mu_\omega
\le
\frac{2-\eta}{\eta}
\int_{F_j}f\omega_j\,\mathrm d\mu
\le
\frac{2-\eta}{\eta}[\omega,v]_{A_1}
\int_{F_j}f\,\mathrm d\mu_v.
\]
Summing over the pairwise disjoint sets \(F_j\)
proves part~\ref{it:weakcoeff2}.
\end{proof}

When \(r\le p\), we also need to control how many stopped averages belong to the same coefficient layer at each point.
The following lemma controls large overlaps within each layer and is the main new ingredient in estimating the tail of the layer decomposition.
\begin{lemma}[Weighted counting estimate]
\label{lem:weak-counting}
Let \(\omega\) be a weight.
For each \(j\ge0\), let \(D_j\in\mathcal F_{\tau_j}\) satisfy
\(D_j\subseteq E_j\).
Suppose that \(\sum_{j\ge0}\omega(D_j)<\infty\).
Then, for every \(s>1\),
\begin{equation*}
\Bigl\|\sum_{j\ge0}\mathbf1_{D_j}\Bigr\|_{L^s_\omega}^s
\lesssim_s
\eta^{1-s}[\omega]_{A_\infty}^{s-1}
\sum_{j\ge0}\omega(D_j).
\end{equation*}
\end{lemma}

\begin{proof}
The weighted conditional expectation identity
\eqref{eq: weighted cond exp identity} and the linear packing
estimate \eqref{eq: linear sparse sum controlled by Ainfty} give
\[
\E_\omega\Bigl[
\sum_{k\ge j}\mathbf1_{D_k}
\,\Bigm|\,\mathcal F_{\tau_j}
\Bigr]
\le
\frac1{\omega_j}
\E\Bigl[
\sum_{k\ge j}\omega_k\mathbf1_{E_k}
\,\Bigm|\,\mathcal F_{\tau_j}
\Bigr]
\le\eta^{-1}[\omega]_{A_\infty}.
\]
Applying Lemma~\ref{lemma: principal lemma conditional version}
under \(\mu_\omega\), with coefficients
\(\mathbf1_{D_k}\) and \(w=1\), and integrating over \(X\) yields
\[
\int_X\Bigl(\sum_{j\ge0}\mathbf1_{D_j}
\Bigr)^s\,\mathrm d\mu_\omega
\lesssim_s
\sum_{j\ge0}\int_{D_j}
\Bigl(
\E_\omega\Bigl[
\sum_{k\ge j}\mathbf1_{D_k}
\,\Bigm|\,\mathcal F_{\tau_j}
\Bigr]
\Bigr)^{s-1}\,\mathrm d\mu_\omega
\le
\eta^{1-s}[\omega]_{A_\infty}^{s-1}
\sum_{j\ge0}\omega(D_j).\qedhere
\]
\end{proof}

\subsection{Proof of the weak-type theorem}
\label{sec:weak-type-proofs}

For \(p>r\), Lemma~\ref{lem:weak-counting} controls the tail of the coefficient layer decomposition below, while the remaining layers can be summed geometrically.
For square functions, the earlier direct argument in the dyadic setting in \cite[Section~4]{DomingoSalazarLaceyRey2016} incurs an additional logarithmic factor.
This loss can be avoided either by testing estimates, as in \cite[Theorem~1.2]{hytonen_weak_2017}, or by our counting estimate.
Our counting argument additionally covers the endpoint \(p=1\) when \(0<r<1\), which is not possible using the testing approach of  \cite{hytonen_weak_2017}.

\begin{proof}[Proof of Theorem~\ref{thm: sparse operators weak type bound}]
We first prove part \ref{item: sparse operators weak type bound p not 1}. Let \(f\in L^p_\sigma\) be nonnegative. By homogeneity, we shall see that it suffices to estimate \[\omega(\{\mathcal A_{\mathcal S}^r(f\sigma)>1\}).\]
Let \(D_j^m\) be the coefficient layers from Lemma~\ref{lem:weak-coefficient-layer}\ref{it:weakcoeff1}. The lower bound on each layer and Lemma~\ref{lem:weak-coefficient-layer}\ref{it:weakcoeff1} give
\begin{equation}
\label{eq:weak-layer-mass}
\sum_{j\ge0}\omega(D_j^m)
\lesssim_{p,\eta}
\left(1-\frac\eta2\right)^{-mp}
[\omega,\sigma]_{A_p}\|f\|_{L^p_\sigma}^p.
\end{equation}
Thus \(\sum_{j\ge0}\omega(D_j^m)<\infty\) for every \(m\).
On \(\{M_{1,\mathcal S}(f\sigma)\le1\}\), every nonzero coefficient belongs to exactly one layer. Using the upper bound on each layer and Lemma~\ref{lem:weak-sequential-maximal}\ref{item:weak-sequential-maximal-p-general} we obtain
\begin{equation}
\label{eq:weak-layer-reduction}
\begin{aligned}
\omega\bigl(\{\mathcal A_{\mathcal S}^r(f\sigma)>1\}\bigr)
&\le [\omega,\sigma]_{A_p}\|f\|_{L^p_\sigma}^p+
\omega\Bigl(
\Bigl\{
\sum_{m\ge0}\left(1-\frac\eta2\right)^{mr}
\sum_{j\ge0}\mathbf1_{D_j^m}>1
\Bigr\}
\Bigr).
\end{aligned}
\end{equation}

If \(r>p\), Chebyshev's inequality,
\eqref{eq:weak-layer-reduction}, and \eqref{eq:weak-layer-mass} give
\begin{equation}
\label{eq:weak-above-diagonal-level}
\omega\bigl(\{\mathcal A_{\mathcal S}^r(f\sigma)>1\}\bigr)
\lesssim_{p,\eta}
[\omega,\sigma]_{A_p}\|f\|_{L^p_\sigma}^p
\sum_{m\ge0}\left(1-\frac\eta2\right)^{m(r-p)}
=
\frac{[\omega,\sigma]_{A_p}\|f\|_{L^p_\sigma}^p}
{1-\left(1-\frac\eta2\right)^{r-p}}.
\end{equation}
However, for \(0<r\le p\), this geometric series diverges.
We therefore use the counting estimate in Lemma~\ref{lem:weak-counting} for the layers with large \(m\).
Take \(s=2p/r\), so that \(r-p/s=r/2>0\). 
Lemma~\ref{lem:weak-counting} and \eqref{eq:weak-layer-mass} give
\[
\Bigl\|\sum_{j\ge0}\mathbf1_{D_j^m}\Bigr\|_{L^s_\omega}^s
\lesssim_{p,r,\eta}
[\omega,\sigma]_{A_p}\|f\|_{L^p_\sigma}^p
[\omega]_{A_\infty}^{s-1}
\Bigl(1-\frac\eta2\Bigr)^{-mp}.
\]
Choose \(m_0\) to be the smallest nonnegative integer such that
\(
(1-\frac\eta2)^{m_0r}\le [\omega]_{A_\infty}^{-1}.
\)
By the choice of \(m_0\),
\[
\Bigl(1-\frac\eta2\Bigr)^{m_0r}\simeq_{r,\eta}[\omega]_{A_\infty}^{-1}\qquad\text{and}
\qquad m_0\lesssim_{r,\eta}\log(e+[\omega]_{A_\infty}).
\]

Splitting the sum in \eqref{eq:weak-layer-reduction} at \(m_0\)
gives
\[
\begin{aligned}
\omega\bigl(\{\mathcal A_{\mathcal S}^r(f\sigma)>1\}\bigr)
\le [\omega,\sigma]_{A_p}\|f\|_{L^p_\sigma}^p
&+
\omega\Bigl(\Bigl\{
\sum_{m<m_0}\Bigl(1-\frac\eta2\Bigr)^{mr}
\sum_{j\ge0}\mathbf1_{D_j^m}>\frac12
\Bigr\}\Bigr)\\
&+
\omega\Bigl(\Bigl\{
\sum_{m\ge m_0}\Bigl(1-\frac\eta2\Bigr)^{mr}
\sum_{j\ge0}\mathbf1_{D_j^m}>\frac12
\Bigr\}\Bigr).\\
&=: [\omega,\sigma]_{A_p}\|f\|_{L^p_\sigma}^p + \mathrm{I} + \mathrm{II}
\end{aligned}
\]
For I, we have by Chebyshev's inequality and
\eqref{eq:weak-layer-mass} 
\begin{align*}
\mathrm{I}
&\lesssim_{p,\eta}
[\omega,\sigma]_{A_p}\|f\|_{L^p_\sigma}^p
\sum_{m<m_0}\left(1-\frac\eta2\right)^{m(r-p)}
\lesssim_{p,r,\eta}
[\omega,\sigma]_{A_p}\|f\|_{L^p_\sigma}^p
\begin{cases}
[\omega]_{A_\infty}^{p/r-1},&r<p,\\[1mm]
\log(e+[\omega]_{A_\infty}),&r=p,
\end{cases}
\end{align*}
where we used that when \(r=p\), the sum equals \(m_0\), yielding the logarithmic bound.
For II we have by Chebyshev's and Minkowski's
inequalities
\begin{align*}
\mathrm{II}
&\lesssim_s
\Bigl(
\sum_{m\ge m_0}\left(1-\frac\eta2\right)^{mr}
\Bigl\|\sum_{j\ge0}\mathbf1_{D_j^m}\Bigr\|_{L^s_\omega}
\Bigr)^s\\*
&\lesssim_{p,r,\eta}
[\omega,\sigma]_{A_p}\|f\|_{L^p_\sigma}^p
[\omega]_{A_\infty}^{s-1}
\Bigl(\sum_{m\ge m_0}\left(1-\frac\eta2\right)^{mr/2}\Bigr)^s\\*
&\lesssim_{p,r,\eta}
[\omega,\sigma]_{A_p}\|f\|_{L^p_\sigma}^p
[\omega]_{A_\infty}^{s-1}\left(1-\frac\eta2\right)^{m_0p}\\*
\lesssim_{p,r,\eta}
[\omega,\sigma]_{A_p}\|f\|_{L^p_\sigma}^p
[\omega]_{A_\infty}^{p/r-1}.
\end{align*}
The last inequality uses the choice of \(m_0\) and \(s=2p/r\).
Since \([\omega]_{A_\infty}\ge1\), the two estimates and \eqref{eq:weak-layer-reduction} give the required level-set bound for \(r\le p\).
Applying these bounds and \eqref{eq:weak-above-diagonal-level} to \(f/\lambda\), for \(\lambda>0\), gives the bounds at level \(\lambda\), proving part~\ref{item: sparse operators weak type bound p not 1}.

For part~\ref{thm:sparse-weak-endpoint}, take \(f\ge0\) in \(L^1_v\) and repeat the
argument with \(p=1\), replacing \(f\sigma\) by \(f\) and
\([\omega,\sigma]_{A_p}\|f\|_{L^p_\sigma}^p\) by
\([\omega,v]_{A_1}\|f\|_{L^1_v}\).
Use the layers from Lemma~\ref{lem:weak-coefficient-layer}\ref{it:weakcoeff2}
and the maximal estimate from Lemma~\ref{lem:weak-sequential-maximal}\ref{item:weak-sequential-maximal-p-1}.
When \(r\le1\), apply Lemma~\ref{lem:weak-counting} with \(s=2/r\).
The same calculations and homogeneity prove part~\ref{thm:sparse-weak-endpoint}.
\end{proof}

\section{Applications}
\label{sec:applications}
In this section we apply our weighted estimates for sparse operators to Doob's
maximal operator and the Rubio de Francia square function.

\subsection{Sparse domination of Doob's maximal operator}
\label{sec:doob-maximal}
In this subsection we prove Corollary~\ref{cor: max op ap ainfty bound}. We first prove a sparse bound for the maximum over finitely many times. Since \(\mu\) is \(\sigma\)-finite, for each \(f\) the essential supremum defining \(Mf\) can be taken over a countable set of times.
The strong and weak estimates for \(M\) then follow from Theorems~\ref{thm: sparse operators ap ainfty bound} and \ref{thm: sparse operators weak type bound}, respectively, by passage to the limit.
For \(t_1<\cdots<t_N\), set
\[
M_{t_1,\ldots,t_N}f:=\max_{1\le k\le N}\E[|f|\mid\mathcal F_{t_k}].
\]
The following proposition gives a pointwise sparse bound for this finite-time maximal operator for every \(r>0\). We will use it with \(r=p\) for the strong-type estimate and with \(r>p\), for instance \(r=2p\), for the weak-type estimate.
\begin{proposition}
\label{prop max op ptw by sparse}
Let \(1\le s<\infty\), let \(t_1<\cdots<t_N\), and let \(f\in L^s(X,\mu)\).
Then there exists a \(\tfrac12\)-sparse sequence \(\mathcal S_N=\{\tau_j\}_{j=0}^\infty\), taking values in \(\{t_1,\ldots,t_N,\infty\}\), such that, for all \(r>0\),
\[
M_{t_1,\ldots,t_N}f \le 2\mathcal A_{\mathcal S_N}^r(f)
\qquad\mu\text{-almost everywhere.}
\]
\end{proposition}

\begin{proof}
Apply Lemma~\ref{lem:weighted-principal-stopping-sequence} with weight \(1\) to \(|f|\) along the sequence \(t_1,\ldots,t_N,t_N,\ldots\). Let \(\mathcal S_N=\{\tau_j\}_{j=0}^\infty\) be the resulting principal sequence. By the lemma, \(\mathcal S_N\) is \(\tfrac12\)-sparse.

For all \(1\le k\le N\), the estimate in Lemma~\ref{lem:weighted-principal-stopping-sequence} gives, almost everywhere, either
\(\E[|f|\mid\mathcal F_{t_k}]=0\), or
\[
\E[|f|\mid\mathcal F_{t_k}]
\le
2\E[|f|\mid\mathcal F_{\tau_j}]
\le
2\mathcal A_{\mathcal S_N}^r(f)
\]
for some \(j\). Taking the maximum over \(1\le k\le N\) proves the result.
\end{proof}
We now pass from the finite-time estimate to the full maximal operator.
\begin{proof}[Proof of Corollary~\ref{cor: max op ap ainfty bound}]
Let \(1\le p<\infty\). For \(p>1\), let \(\tfrac g\sigma\in L^p_\sigma\), while for \(p=1\), let \(g\in L^1_v\).
Since \(\mu\) is \(\sigma\)-finite, the essential supremum of any family of nonnegative measurable functions agrees \(\mu\)-almost everywhere with the supremum of a countable subfamily.
Thus we may choose times \(t_k\in\mathbb R\), depending on \(g\), such that
\[
M(g)
=
\sup_{k\ge1}\E[|g|\mid\mathcal F_{t_k}]
\qquad\mu\text{-almost everywhere.}
\]
It therefore suffices to prove the estimates for finite maxima, with constants independent of the number and choice of times.

For \(m\ge1\), define
\[
g_m:=
\begin{cases}
g\mathbf1_{\{\sigma\le m\}}, & p>1,\\
g\mathbf1_{\{v\ge m^{-1}\}}, & p=1.
\end{cases}
\]
Then \(g_m\in L^p(X,\mu)\), so we can apply Proposition~\ref{prop max op ptw by sparse} with \(s=p\) and \(f=g_m\), which gives a \(\tfrac12\)-sparse sequence \(\mathcal S_N\), which may depend on \(m\).
For \(p>1\), taking \(r=p\) gives
\[
\|M_{t_1,\ldots,t_N}g_m\|_{L^p_\omega}
\le
2\|\mathcal A_{\mathcal S_N}^p(g_m)\|_{L^p_\omega}
\lesssim_p
[\omega,\sigma]_{A_p}^{1/p}
\bigl([\sigma]_{A_\infty}^{1/p}+1\bigr)
\Bigl\|\frac{g_m}{\sigma}\Bigr\|_{L^p_\sigma}
\lesssim_p
[\omega,\sigma]_{A_p}^{1/p}
[\sigma]_{A_\infty}^{1/p}
\Bigl\|\frac g\sigma\Bigr\|_{L^p_\sigma}.
\]
Here we used Theorem~\ref{thm: sparse operators ap ainfty bound} and \([\sigma]_{A_\infty}\ge1\).
The same sparse sequence gives the pointwise bound with
\(r=2p\). Applying
Theorem~\ref{thm: sparse operators weak type bound} yields
\[
\|M_{t_1,\ldots,t_N}g_m\|_{L^{p,\infty}_\omega}
\lesssim_p
\begin{cases}
[\omega,\sigma]_{A_p}^{1/p}
\bigl\|\frac g\sigma\bigr\|_{L^p_\sigma}, & p>1,\\[1mm]
[\omega,v]_{A_1}\|g\|_{L^1_v}, & p=1.
\end{cases}
\]
Both estimates are uniform in \(m\) and \(N\).

Since \(|g_m|\uparrow |g|\),
conditional monotone convergence gives
\[
M_{t_1,\ldots,t_N}g_m
\uparrow
M_{t_1,\ldots,t_N}g.
\]
First let \(m\to\infty\), and then let \(N\to\infty\).
The strong estimate passes to the limit by monotone convergence,
and the weak estimate by continuity from below of the level sets
at each \(\lambda>0\).
\end{proof}
\subsection{Application to the Rubio de Francia square function}
\label{sec:rubio-weak-endpoint}
In this subsection we apply Theorem~\ref{thm: sparse operators weak type bound} to obtain a mixed two weight weak-type \((2,2)\) estimate for the Rubio de Francia square function. In the one weight setting, this removes the logarithmic factor in \cite[Corollary 1.6]{GRS21}. We work on \(\mathbb R\) with Lebesgue measure and a dyadic filtration. For a collection \(\mathscr I\) of pairwise disjoint frequency intervals, define
\[
R_{\mc I}f
:=\Bigl(
\sum_{I\in\mc I}
\bigl|\mathcal F^{-1}(\mathbf1_I\mathcal F(f))\bigr|^2
\Bigr)^{1/2},\]
where
\[
\mathcal Ff(\xi):
=\int_{\mathbb R}f(x)e^{-2\pi i x\xi}\,\mathrm dx
\qquad\text{and}\qquad
\mathcal F^{-1}g(x)
:=\int_{\mathbb R}g(\xi)e^{2\pi i x\xi}\,\mathrm d\xi.
\]

\begin{proof}[Proof of Corollary~\ref{cor:rubio-weak-endpoint}]
For compactly supported \(f\in L^2(\mathbb R)\), \cite[Theorem~A]{DFPR25} and the three lattice lemma reduce the proof to estimating sparse operators on three shifted dyadic grids.
Applying
part~\ref{thm:sparse-weak-endpoint}
of Theorem~\ref{thm: sparse operators weak type bound}
with \(r=\tfrac12\) gives
\[
\Bigl\|
\sum_{Q\in\mathcal S}
\ip{|f|^2}_Q^{1/2}\mathbf1_Q
\Bigr\|_{L^{2,\infty}_\omega}
=
\bigl\|\mathcal A_{\mathcal S}^{1/2}(|f|^2)
\bigr\|_{L^{1,\infty}_\omega}^{1/2}\lesssim
[\omega,v]_{A_1}^{1/2}
[\omega]_{A_\infty}^{1/2}
\|f\|_{L^2_v}.
\]
The result follows by density.
\end{proof}

\section{Finite-complexity dyadic sparse forms}
\label{sec:finite-complexity}

We now turn to the finite-complexity dyadic setting in \(\mathbb R^d\). Let {\(\mu\) be a locally finite Borel measure on \(\mathbb R^d\), and let} $\ms{D}$ be a dyadic lattice, see \cite[Definition 2.1]{LernerNazarov2019}. A weight $\sigma$ is a {locally $\mu$-integrable function that is positive \(\mu\)-almost everywhere}. As before, we write \(L^p_\sigma=L^p(\mathbb R^d,\mu_\sigma)\) and use \(\sigma(Q)\) for the weighted measure of \(Q\). We set \(\langle f\rangle_Q^\sigma:=\sigma(Q)^{-1}\int_Q f\,\mathrm d\mu_\sigma\).

{For \(Q\in\mathscr D\) and \(r\in\mathbb N_0\), let \(Q^{(r)}\) denote the \(r\)-th dyadic ancestor of \(Q\), with \(Q^{(0)}=Q\). If \(Q,P\in\mathscr D\) share a common dyadic ancestor, define
\[
\distD(Q,P):=\min\{r+t:r,t\in\mathbb N_0,\ Q^{(r)}=P^{(t)}\},
\]
and set \(\distD(Q,P)=\infty\) otherwise. A relation \(\Gamma\subseteq\mathscr D\times\mathscr D\) has complexity at most \(\kappa\in\mathbb N_0\) if \(\distD(Q,P)\le\kappa\) for every \((Q,P)\in\Gamma\).

For \(\Gamma\subseteq\mathscr D\times\mathscr D\) and \(\mathcal S\subseteq\mathscr D\), recall that}
\[
\Chat_{\mathcal S,\Gamma}(f,g)
:=
\sum_{\substack{Q,P\in\mathcal S:(Q,P)\in\Gamma}}
\avg{f}{Q}\avg{g}{P}\,c_{Q,P},
\qquad
c_{Q,P}:=\min\{\mu(Q),\mu(P)\}.
\]
The main difficulty compared with classical sparse forms is that {the two averages may be taken over different cubes.} Finite complexity allows us to recover enough structure: we first decompose the relation into finitely many branches on which each cube has a unique partner and the relevant nesting is controlled.

Afterwards, we decompose $\Chat_{\mathcal S,\Gamma}$ into these branches and for each branch, we then associate a positive operator and characterize its two weight norm by direct and dual testing conditions. The proof follows the classical principal-cube strategy, {with an additional geometric argument accounting for the fact that the principal cubes associated with the two averages need not be nested.}

\subsection{Branch decomposition}
We start by decomposing a relation \(\Gamma\subseteq\ms{D}\times\ms{D}\) into pieces where each cube occurs at most once in each coordinate, which we will call branches. More precisely, a relation
\(\Gamma_0\subseteq\ms{D}\times\ms{D}\) is a \emph{branch} if it has the form
\[
\Gamma_0=\cbrace{(Q,P_Q):Q\in\mc{S}_0},
\]
where \(\mc{S}_0\subseteq\ms{D}\) and the map \(Q\mapsto P_Q\) is injective.
Writing \(H_Q\) for the minimal common dyadic ancestor of \(Q\) and \(P_Q\), we additionally require that
\begin{equation}
\label{eq:branch-nesting}
Q\subsetneq R\Longrightarrow H_Q\subseteq R,
\qquad
P_Q\subsetneq P_R\Longrightarrow H_Q\subseteq P_R
\end{equation}
for all \(Q,R\in\mc{S}_0\).
Thus, if \(Q\) lies strictly inside \(R\), then its partner \(P_Q\)
must also lie inside \(R\), and the same holds with the two coordinates
interchanged.

\begin{proposition}
\label{prop:finite-complexity-branch-decomposition}
Let \(\Gamma\subseteq\ms{D}\times\ms{D}\) have complexity at most
\(\kappa\). There is a constant \(M_\kappa\), only depending on $\kappa$ and $d$, and branches \(\Gamma_1,\ldots,\Gamma_{M_\kappa}\) such that
\[
\Gamma
=
\bigsqcup_{j=1}^{M_\kappa}\Gamma_j.
\]
\end{proposition}

\begin{proof}
Let $m_{\kappa}$ be the number of cubes at dyadic distance at most \(\kappa\) from a fixed cube, which only depends on $\kappa$ and $d$.
For each fixed \(Q\in \ms{D}\), number the cubes \(P\in \ms{D}\) such that \((Q,P)\in\Gamma\). Independently, for each fixed \(P\in\ms D\), number the cubes \(Q\in\ms D\) such that \((Q,P)\in\Gamma\). Thus every pair \((Q,P)\) receives an ordered pair of numbers. At most \(m_{\kappa}\) numbers are needed on either side, so grouping pairs with the same ordered pair produces at most \(m_{\kappa}^2\) groups.

Within one group, if \((Q,P)\) and \((Q,P')\) have the same first coordinate, then \(P\) and \(P'\) have the same number in the enumeration of the partners of \(Q\), and hence \(P=P'\). Similarly, if \((Q,P)\) and \((Q',P)\) have the same second coordinate, then \(Q=Q'\). Thus each group is the graph of an injective map.

Next, split each group according to the generation of \(Q\) modulo \(\kappa+1\) and the generation of \(P\) modulo \(\kappa+1\). Each subgroup is still the graph of an injective map, and every original group gives at most \((\kappa+1)^2\) subgroups. Thus the total number of nonempty groups $M_{\kappa}$ is at most $(\kappa+1)^2 m_{\kappa}^2$.
Denote these groups by $\Gamma_1,\ldots,\Gamma_{M_\kappa}.$
For each \(j\), let \(\mc{S}_j\) be the set of first coordinates
in \(\Gamma_j\), and write \(P_Q\) for the unique partner of
\(Q\in\mc{S}_j\). Then
\[
\Gamma_j=\cbrace{(Q,P_Q):Q\in\mc{S}_j},
\]
and \(Q\mapsto P_Q\) is injective.

Fix \(Q,R\in\mathcal S_j\). If \(P_Q\subsetneq P_R\), then their generations differ by a multiple of \(\kappa+1\). Hence \(P_Q\) lies at least \(\kappa+1\) generations below \(P_R\). Since \(\distD(Q,P_Q)\le\kappa\), the cube \(H_Q\) lies at most \(\kappa\) generations above \(P_Q\). Hence \(H_Q\subseteq P_R\). The same argument shows that \(Q\subsetneq R\) implies \(H_Q\subseteq R\). Thus each \(\Gamma_j\) is a branch.
\end{proof}

\subsection{Testing condition on a branch}
Fix a branch
\[
\Gamma_j=\cbrace{(Q,P_Q):Q\in\mc{S}_j}.
\]
Given weights \(\omega\) and \(\sigma\), define the positive operator
\begin{equation}\label{eq:branchoperator}
\mathsf T_j^\sigma f
:=
\sum_{Q\in\mc{S}_j}
\frac{c_{Q,P_Q}}{\mu(P_Q)}
\avg{\sigma}{Q}\,\avgsig{f}{Q}\,\mathbf 1_{P_Q}.
\end{equation}
For \(R\in\mathscr D\), define the localized direct and dual testing sums by
\[
\mathsf T_{j,R}(\sigma)
:=
\sum_{\substack{Q\in\mathcal S_j\\Q\subseteq R}}
\frac{c_{Q,P_Q}}{\mu(P_Q)}
\avg{\sigma}{Q}\mathbf 1_{P_Q}
\qquad
\text{and}
\qquad
\mathsf T_{j,R}^*(\omega)
:=
\sum_{\substack{Q\in\mathcal S_j\\P_Q\subseteq R}}
\frac{c_{Q,P_Q}}{\mu(Q)}
\avg{\omega}{P_Q}\mathbf 1_Q.
\]
Note that the direct localization uses \(Q\subseteq R\), whereas the dual localization
uses \(P_Q\subseteq R\). The main result of this section is the following proposition.

\begin{proposition}
\label{prop:finite-complexity-testing} 
Let \(\Gamma_j\subseteq \ms{D} \times \ms{D}\) be a branch,
let \(1<p<\infty\), and let \(\sigma,\omega\) be weights. Set
\[
\mathfrak T_{p,j}
:=
\sup_{R\in\mathscr D}
\frac{\|\mathsf T_{j,R}(\sigma)\|_{L^p_\omega}}
{\sigma(R)^{1/p}},
\qquad\text{and}\qquad
\mathfrak T_{p,j}^*
:=
\sup_{R\in\mathscr D}
\frac{\|\mathsf T_{j,R}^*(\omega)\|_{L^{p'}_\sigma}}
{\omega(R)^{1/p'}}.
\]
Then
\[
\|\mathsf T_j^\sigma\|_{L^p_\sigma\to L^p_\omega}
\simeq_{p}
\mathfrak T_{p,j}+\mathfrak T_{p,j}^* .
\]
\end{proposition}

Note that the branch in Proposition \ref{prop:finite-complexity-testing} does not need to
have finite complexity.
When \[\Gamma_j \subseteq \cbrace{(Q,Q):Q\in \ms{D}},\] Proposition \ref{prop:finite-complexity-testing} is the classical diagonal testing theorem, see \cite[Theorem~6.1]{HytonenA2remarks}, \cite[Theorem~1.11]{lacey_two_2010} with \(q=p\), and \cite[Theorem~2.1]{Treil}.

For the sufficiency proof of Proposition \ref{prop:finite-complexity-testing}, we will construct principal
collections \(\mathcal F\) and \(\mathcal G\) from \(f\) and \(g\).
{For a collection \(\mathcal A\subseteq\mathscr D\), we write
\(\pi_{\mathcal A}(Q)\) for the smallest cube in \(\mathcal A\)
containing \(Q\), whenever such a cube exists.}
When \(P_Q=Q\), the cubes \(\pi_{\mathcal F}(Q)\) and \(\pi_{\mathcal G}(P_Q)\) are nested, since
they both contain \(Q\). On a general branch, they need not be nested.
The following lemma replaces this nesting by a weaker geometric
alternative. We can split $\mc{S}_j$ into two families of cubes so that each \(\pi_{\mathcal G}(P_Q)\) is paired with at most $2$ cubes \(\pi_{\mathcal F}(Q)\) in the first family, and each \(\pi_{\mathcal F}(Q)\) is paired with at most $2$ cubes \(\pi_{\mathcal G}(P_Q)\) in the second family.

\begin{lemma}
\label{lem:stopping-geometry}
Let \(\Gamma_j = \cbrace{(Q,P_Q):Q\in\mathcal S_j}\) be a branch.
Let
\[
\mathcal F\subseteq\mathcal S_j,
\qquad
\mathcal G\subseteq\cbrace{P_Q:Q\in\mathcal S_j},
\]
and suppose that \(\pi_{\mathcal F}(Q)\) and \(\pi_{\mathcal G}(P_Q)\) are defined for every \(Q\in\mathcal S_j\).
Then \(\mc{S}_j\) can be partitioned as
\[
\mc{S}_j=\mathcal Q_{\mathcal F}\mathbin{\dot{\cup}}\mathcal Q_{\mathcal G}
\]
such that 
\begin{align*}
\#\{\pi_{\mathcal F}(Q):Q\in\mathcal Q_{\mathcal F},\ \pi_{\mathcal G}(P_Q)=G\}&\leq 2, &&G\in\mathcal G,\\
\#\{\pi_{\mathcal G}(P_Q):Q\in\mathcal Q_{\mathcal G},\ \pi_{\mathcal F}(Q)=F\}&\leq 2, && F\in\mathcal F.
\end{align*}
\end{lemma}

\begin{proof}
The sets
\[
\begin{aligned}
\mathcal Q_{\mathcal G}:=\cbrace{Q\in\mathcal S_j:\pi_{\mathcal F}(Q)\subsetneq \pi_{\mathcal G}(P_Q)},
\qquad
\mathcal Q_{\mathcal F}:=\mathcal S_j\setminus\mathcal Q_{\mathcal G}
\end{aligned}
\]
form a partition of \(\mathcal S_j\).
Fix \(G\in\mathcal G\) and let \(R\in\mathcal S_j\) be the unique cube such that \(P_R= G\).
Let \(Q\in\mathcal Q_{\mathcal F}\) satisfy \(\pi_{\mathcal G}(P_Q)=G\). Suppose first that \(Q\ne R\). Then \(P_Q\subsetneq P_R\) by injectivity, so \eqref{eq:branch-nesting} gives \(Q\subseteq G\). Since \(\pi_{\mathcal F}(Q)\) and \(G\) both contain \(Q\), they are nested, so the definition of \(\mathcal Q_{\mathcal F}\) gives \(G\subseteq \pi_{\mathcal F}(Q)\).
If there were an \(F\in\mathcal F\) with
\[
G\subseteq F\subsetneq \pi_{\mathcal F}(Q),
\]
then \(F\) would contain \(Q\), contradicting the minimality of \(\pi_{\mathcal F}(Q)\). Thus \(\pi_{\mathcal F}(Q)\) is the smallest member of \(\mathcal F\) containing \(G\). Hence all such \(\pi_{\mathcal F}(Q)\) for \(Q\ne R\) coincide, and the index \(Q=R\) gives at most one further possibility.
The second bound follows similarly. 
\end{proof}

We are now ready to prove the testing characterization.

\begin{proof}[Proof of Proposition~\ref{prop:finite-complexity-testing}]
For \(R\in\mathscr D\), positivity and duality give
\[
\begin{aligned}
\|\mathsf T_{j,R}(\sigma)\|_{L^p_\omega}
&\le \|\mathsf T_j^\sigma\|\,\sigma(R)^{1/p},\\
\|\mathsf T_{j,R}^*(\omega)\|_{L^{p'}_\sigma}
&\le \|\mathsf T_j^\sigma\|\,\omega(R)^{1/p'}.
\end{aligned}
\]
Taking the supremum over \(R\) proves necessity.

For sufficiency, write \(\Gamma_j=\cbrace{(Q,P_Q):Q\in\mathcal S_j}\). By positivity and monotone convergence, we may assume without loss of generality that $\mc{S}_j$ is finite. By duality it suffices to show that for all nonnegative \(f\in L^p_\sigma\) and \(g\in L^{p'}_\omega\)
\begin{equation}\label{eq:formtarget}
\sum_{Q \in\mc{S}_j}\frac{c_{Q,P_Q}}{\mu(P_Q)}\avg{\sigma}{Q}\,\avgsig{f}{Q}\,\langle g\rangle_{P_Q}^{\omega}\omega(P_Q)
\lesssim_p \bigl(\mathfrak T_{p,j}+\mathfrak T_{p,j}^*\bigr)\|f\|_{L^p_\sigma}\|g\|_{L^{p'}_\omega}.
\end{equation}
Construct the collection \(\mathcal F\) as follows. Start with the maximal members of \(\mc{S}_j\). Having chosen \(F\in\mathcal F\), add to \(\mathcal F\) all maximal cubes \(Q\in\mc{S}_j\) such that $Q\subsetneq F$ and
\[
\avgsig{f}{Q}>2\avgsig{f}{F},
\]
and iterate this procedure. Construct \(\mathcal G\) analogously from \(\{P_Q:Q\in\mc{S}_j\}\) using the \(\omega\)-averages of \(g\). {By construction, for every \(Q\in\mc{S}_j\), we have}
\begin{equation}
\label{eq:principal-control}
\avgsig{f}{Q}\le 2\avgsig{f}{\pi_{\mathcal F}(Q)},
\qquad
\langle g\rangle_{P_Q}^{\omega}\le 2\langle g\rangle_{\pi_{\mathcal G}(P_Q)}^{\omega},
\end{equation}
{Applying Lemma~\ref{lemma: carleson embedding} to the dyadic filtration, with exponents \(p\) and \(p'\), gives}
\begin{equation}
\label{eq:principal-carleson-embeddings}
\begin{aligned}
\sum_{F\in\mathcal F}(\avgsig{f}{F})^p\sigma(F)
&\lesssim_p \|f\|_{L^p_\sigma}^p,\\
\sum_{G\in\mathcal G}(\langle g\rangle_G^\omega)^{p'}\omega(G)
&\lesssim_p \|g\|_{L^{p'}_\omega}^{p'}.
\end{aligned}
\end{equation}

By construction, \(\mathcal F\subseteq\mc{S}_j\) and \(\mathcal G\subseteq\cbrace{P_Q:Q\in\mc{S}_j}\). Applying Lemma~\ref{lem:stopping-geometry} to \(\Gamma_j\), we obtain a partition
\(
\mc{S}_j=\mathcal Q_{\mathcal F}\mathbin{\dot\cup}\mathcal Q_{\mathcal G}.
\)

\proofpart{The \(\mathcal Q_{\mathcal F}\)-contribution}
For \(F\in\mathcal F\), set
\[
g_F:=\sum_{G\in\mathcal G_F}\langle g\rangle_G^\omega\mathbf 1_G,\qquad
\mathcal G_F:=\{\pi_{\mathcal G}(P_Q):Q\in\mathcal Q_{\mathcal F},\ \pi_{\mathcal F}(Q)=F\}.
\]
For fixed \(F\), the cubes in \(\mathcal G_F\) containing a given point $x \in \R^d$ form a nested chain. If
\(G'\subsetneq G\) are consecutive cubes in this chain, then
\(\langle g\rangle_{G'}^\omega>2\langle g\rangle_G^\omega\) by construction. Hence, by a geometric series argument
\[
g_F=\sum_{G\in\mathcal G_F}\langle g\rangle_G^\omega \mathbf 1_G\le2\sup_{\substack{G\in\mathcal G_F}}\langle g\rangle_G^\omega\mathbf 1_G,
\]
and therefore
\[
\|g_F\|_{L^{p'}_\omega}^{p'}
\lesssim_p
\sum_{G\in\mathcal G_F}
(\langle g\rangle_G^\omega)^{p'}\omega(G).
\]
By Lemma~\ref{lem:stopping-geometry}, each \(G\in\mathcal G\) belongs to \(\mathcal G_F\) for at most two values of \(F\). Summing over \(F\) and using \eqref{eq:principal-carleson-embeddings}, we obtain
\[
\sum_{F\in\mathcal F}\|g_F\|_{L^{p'}_\omega}^{p'}
\lesssim_p \|g\|_{L^{p'}_\omega}^{p'}.
\]
For \(Q\in\mathcal Q_{\mathcal F}\) with \(\pi_{\mathcal F}(Q)=F\), we have \(Q\subseteq F\), while
\(g_F\ge\langle g\rangle_{\pi_{\mathcal G}(P_Q)}^\omega\) on \(P_Q\). Therefore, using \eqref{eq:principal-control} and \eqref{eq:principal-carleson-embeddings}, we obtain
\[
\begin{aligned}
\sum_{Q\in\mathcal Q_{\mathcal F}}
\frac{c_{Q,P_Q}}{\mu(P_Q)}
\avg{\sigma}{Q}\,\avgsig{f}{Q}\,
\langle g\rangle_{P_Q}^{\omega}\omega(P_Q)
&\lesssim
\sum_{F\in\mathcal F}\avgsig{f}{F}
\sum_{\substack{Q\in\mathcal Q_{\mathcal F}\\\pi_{\mathcal F}(Q)=F}}
\frac{c_{Q,P_Q}}{\mu(P_Q)}
\avg{\sigma}{Q}
\int_{P_Q}g_F\,\mathrm d\mu_\omega\\
&\le
\sum_{F\in\mathcal F}
\avgsig{f}{F}
\int_{\R^d}\mathsf T_{j,F}(\sigma)g_F\,\mathrm d\mu_\omega\\
&\le
\mathfrak T_{p,j}
\sum_{F\in\mathcal F}
\avgsig{f}{F}\sigma(F)^{1/p}
\|g_F\|_{L^{p'}_\omega}\\
&\lesssim_{p}
\mathfrak T_{p,j}
\|f\|_{L^p_\sigma}
\|g\|_{L^{p'}_\omega}.
\end{aligned}
\]

\proofpart{The \(\mathcal Q_{\mathcal G}\)-contribution}
The reversed relation \(\cbrace{(P_Q,Q):Q\in\mc{S}_j}\) is again a branch. We apply the preceding argument to the pairs indexed by \(\mathcal Q_{\mathcal G}\), using the second bound in Lemma~\ref{lem:stopping-geometry} and interchanging
\[
(f,\sigma,p,\mathcal F)
\quad\text{and}\quad
(g,\omega,p',\mathcal G).
\]
The direct testing constant becomes \(\mathfrak T_{p,j}^*\). Hence
\[
\sum_{Q\in\mathcal Q_{\mathcal G}}\frac{c_{Q,P_Q}}{\mu(P_Q)}\avg{\sigma}{Q}\,\avgsig{f}{Q}\,\langle g\rangle_{P_Q}^{\omega}\omega(P_Q)
\lesssim_p \mathfrak T_{p,j}^*\|f\|_{L^p_\sigma}\|g\|_{L^{p'}_\omega},
\]
proving \eqref{eq:formtarget} and thus finishing the proof.
\end{proof}

\section{Mixed estimates on branches}
\label{sec:finite-complexity-characteristics}
In this section we prove Theorem~\ref{thm: finite complexity mixed intro}. We will first use the testing characterization from Section~\ref{sec:finite-complexity} to obtain mixed bounds on a fixed branch, and then sum these bounds over the branch decomposition obtained in Proposition \ref{prop:finite-complexity-branch-decomposition}.

Throughout this section, fix \(1<p<\infty\) and let \(\sigma,\omega\) be weights. Let \(\mathcal S\subseteq\mathscr D\) be \(\eta\)-sparse with respect to $\mu$, and fix a branch
\[
\Gamma_j=\{(Q,P_Q):Q\in\mathcal S_j\}
\subseteq\mathcal S\times\mathcal S.
\]
\begin{proposition}
\label{prop:branch-mixed-estimate}
The branch operator \(\mathsf T_j^\sigma\) as in \eqref{eq:branchoperator} satisfies
\[
\|\mathsf T_j^\sigma\|_{L^p_\sigma\to L^p_\omega}
\lesssim_{p} \eta^{-1}
[\omega,\sigma]_{\Acal_{p,\Gamma_j}(\mu)}^{1/p}
\Bigl(
[\sigma]_{\AinfD}^{1/p}
+[\omega]_{\AinfD}^{1/p'}
\Bigr).
\]
\end{proposition}
The characteristic \([\omega,\sigma]_{\Acal_{p,\Gamma_j}(\mu)}\) also gives a necessary condition for boundedness:
\[
[\omega,\sigma]_{\Acal_{p,\Gamma_j}(\mu)}^{1/p}
\le
\|\mathsf T_j^\sigma\|_{L^p_\sigma\to L^p_\omega}.
\]
This follows by testing \(\mathsf T_j^\sigma\) on
\(\mathbf1_Q\), retaining the term indexed by \(Q\),
and taking the supremum over \(Q\in\mathcal S_j\).
Proposition \ref{prop:branch-mixed-estimate} will follow from 
Propositions~\ref{prop:finite-complexity-testing} and \ref{prop:forward-testing} below.
The implicit constant in Proposition \ref{prop:branch-mixed-estimate} is independent of the dimension, and \(\Gamma_j\) need not have finite complexity.
In the proof of Theorem~\ref{thm: finite complexity mixed intro}, we will use finite complexity only to decompose the original relation into finitely many branches.
By Proposition~\ref{prop:finite-complexity-testing}, it suffices to prove
\[
\mathfrak T_{p,j}
\lesssim_{p,\eta}
[\omega,\sigma]_{\Acal_{p,\Gamma_j}(\mu)}^{1/p}
[\sigma]_{\AinfD}^{1/p}
\qquad\text{and}\qquad
\mathfrak T_{p,j}^*
\lesssim_{p,\eta}
[\omega,\sigma]_{\Acal_{p,\Gamma_j}(\mu)}^{1/p}
[\omega]_{\AinfD}^{1/p'}.
\]

\subsection{Coefficient and packing estimates}
We first collect some coefficient and packing bounds used for the testing estimates. The dyadic \(A_\infty\)-characteristic controls weighted sparse sums, while the nesting condition in the definition of a branch supplies the additional packing needed when the input and output cubes differ.




To express the direct and dual testing coefficients with a common denominator, for \(Q\in\mathcal S_j\) set\[
\rho_Q:=\frac{\mu(Q)\mu(P_Q)}{c_{Q,P_Q}}
=\max\{\mu(Q),\mu(P_Q)\},
\qquad
\widetilde{\sigma}_Q:=\frac{\sigma(Q)}{\rho_Q},
\qquad
\widetilde{\omega}_Q:=\frac{\omega(P_Q)}{\rho_Q}.
\]
Equivalently,
\[
\widetilde{\sigma}_Q
=\frac{c_{Q,P_Q}}{\mu(P_Q)}\avg{\sigma}{Q},
\qquad
\widetilde{\omega}_Q
=\frac{c_{Q,P_Q}}{\mu(Q)}\avg{\omega}{P_Q}.
\]
Thus \(\widetilde{\sigma}_Q\) and \(\widetilde \omega_Q\) are precisely the
coefficients in \(\mathsf T_{j,R}(\sigma)\) and
\(\mathsf T_{j,R}^*(\omega)\), respectively.
\begin{lemma}
\label{lem:coefficient-estimates}
For every \(Q\in\mathcal S_j\),
\[
(\widetilde{\sigma}_Q)^{p-1}\widetilde{\omega}_Q
\le [\omega,\sigma]_{\Acal_{p,\Gamma_j}(\mu)}\qquad\text{and}\qquad
(\widetilde{\omega}_Q)^{p'-1}\widetilde{\sigma}_Q
\le [\omega,\sigma]_{\Acal_{p,\Gamma_j}(\mu)}^{p'/p}.
\]
\end{lemma}

\begin{proof}
By the definition of the weight characteristic and the normalized coefficients,
\[
\frac{c_{Q,P_Q}^p \omega(P_Q)\sigma(Q)^{p-1}}
{\mu(Q)^p\mu(P_Q)^p}
=
(\widetilde{\sigma}_Q)^{p-1}\widetilde{\omega}_Q
\le [\omega,\sigma]_{\Acal_{p,\Gamma_j}(\mu)}.
\]
This gives the first inequality.
Since \(p'-1=1/(p-1)\), raising this inequality to the power
\(1/(p-1)=p'/p\) gives the second inequality.
\end{proof}

{We next use sparsity and the branch nesting condition
to prove a weighted packing estimate.}
\begin{lemma}[Weighted branch packing]
\label{lem:branch-packing}
Let \(0\le\gamma<1\). For every \(K\in\mathcal S_j\),
\begin{equation*}
\sum_{\substack{Q\in\mathcal S_j\\P_Q\subseteq P_K}}
(\widetilde{\omega}_Q)^\gamma\rho_Q
\le
\frac{1+2\eta^{-1}}{1-\gamma}
(\widetilde{\omega}_K)^\gamma\rho_K.
\end{equation*}
\end{lemma}

\begin{proof}

For all \(R\in\mathscr D\),
\[
\sum_{\substack{Q\in\mathcal S_j\\H_Q\subseteq R}}\rho_Q
\le 2\eta^{-1}\mu(R).
\]
Indeed, \(H_Q\subseteq R\) implies that \(Q,P_Q\subseteq R\). Since \(Q\mapsto P_Q\) is injective and \(\rho_Q\le\mu(Q)+\mu(P_Q)\), the estimate follows from \(\eta\)-sparsity.

Fix \(L {\in\mathcal S_j}\). If \(Q\ne L\) and \(P_Q\subseteq P_L\), injectivity gives \(P_Q\subsetneq P_L\), so \eqref{eq:branch-nesting} gives \(H_Q\subseteq P_L\). Hence
\begin{equation*}
\sum_{P_Q\subseteq P_L}\rho_Q
\le
\rho_L+\sum_{H_Q\subseteq P_L}\rho_Q
\le
\rho_L+2\eta^{-1}\mu(P_L)
\le (1+2\eta^{-1})\rho_L.
\end{equation*}
This proves the result when \(\gamma=0\).

Suppose that \(0<\gamma<1\), and fix \(t>0\). Among the cubes \(P_Q\subseteq P_K\) satisfying \(\widetilde{\omega}_Q>t\), denote the maximal ones by \(P_U\). They are pairwise disjoint, and the estimate for \(\gamma=0\), with \(L=U\), gives
\[
\sum_{\substack{P_Q\subseteq P_K\\
\widetilde{\omega}_Q>t}}\rho_Q
\le
\sum_U\sum_{P_Q\subseteq P_U}\rho_Q
\le (1+2\eta^{-1})\sum_U\rho_U
\le \frac{1+2\eta^{-1}}{t}\sum_U\omega(P_U)
\le \frac{(1+2\eta^{-1})\omega(P_K)}{t}.
\]
The same estimate with \(L=K\) gives the other bound, and hence
\[
\sum_{\substack{P_Q\subseteq P_K\\
\widetilde{\omega}_Q>t}}\rho_Q
\le
(1+2\eta^{-1})\min\Bigl\{\rho_K,\frac{\omega(P_K)}{t}\Bigr\}.
\]

Using \(\omega(P_K)=\widetilde{\omega}_K\rho_K\), integrating this estimate in \(t\), and splitting the integral at \(\widetilde{\omega}_K\), we obtain

\[
\begin{aligned}
\sum_{P_Q\subseteq P_K}
(\widetilde{\omega}_Q)^\gamma\rho_Q
&=
\gamma\int_0^\infty t^{\gamma-1}
\sum_{\substack{P_Q\subseteq P_K\\
\widetilde{\omega}_Q>t}}
\rho_Q\,\mathrm dt 
\\
&\leq
\gamma(1+2\eta^{-1})\Bigl(
\rho_K\int_0^{\widetilde{\omega}_K}t^{\gamma-1}\,\mathrm dt
+
\omega(P_K)
\int_{\widetilde{\omega}_K}^{\infty}t^{\gamma-2}\,\mathrm dt
\Bigr) 
\\
&=(1+2\eta^{-1})\Bigl(1+\frac{\gamma}{1-\gamma}\Bigr)
(\widetilde{\omega}_K)^\gamma\rho_K
\\
&=\frac{1+2\eta^{-1}}{1-\gamma}(\widetilde{\omega}_K)^\gamma\rho_K.
\end{aligned}
\]
This finishes the proof.
\end{proof}

We note that in what follows, the nesting condition in the definition of a branch will only be used through Lemma \ref{lem:branch-packing}.

\subsection{Direct and dual testing}
We will now estimate the direct testing condition using a chain expansion for \(1<p\le2\) and Lemma~\ref{lemma: principal lemma expectation version} for \(p>2\). The dual estimate then follows by reversing the branch and interchanging the weights.
\begin{proposition}
\label{prop:forward-testing}
For every \(R\in\mathscr D\),
\begin{align*}
\|\mathsf T_{j,R}(\sigma)\|_{L^p_\omega}
&\lesssim_{p} \eta^{-1}
[\omega,\sigma]_{\Acal_{p,\Gamma_j}(\mu)}^{1/p}
[\sigma]_{\AinfD}^{1/p}\,\sigma(R)^{1/p},\\
\|\mathsf T_{j,R}^*(\omega)
\|_{L^{p'}_\sigma}
&\lesssim_{p} \eta^{-1}
[\omega,\sigma]_{\Acal_{p,\Gamma_j}(\mu)}^{1/p}
[\omega]_{\AinfD}^{1/p'}\,\omega(R)^{1/p'}.
\end{align*}
\end{proposition}

\begin{proof}
By monotone convergence, it is enough to work with a fixed finite subcollection of \(\{Q\in\mathcal S_j:Q\subseteq R\}\). Unless stated otherwise, all sums below are over this subcollection.

Suppose first that \(1<p\le2\). Telescoping the \(p\)-th powers
of the partial sums along each chain of cubes \(P_Q\), ordered
from largest to smallest, and then integrating gives
\[
\begin{aligned}
\|\mathsf T_{j,R}(\sigma)\|_{L^p_\omega}^p
&\le
p\sum_Q
\widetilde{\sigma}_Q
\Bigl(
\sum_{P_L\supseteq P_Q}\widetilde{\sigma}_L
\Bigr)^{p-1}
\omega(P_Q)\\
&\lesssim_p
[\omega,\sigma]_{\Acal_{p,\Gamma_j}(\mu)}
\sum_Q
(\widetilde{\sigma}_Q)^{2-p}
\Bigl(
\sum_{P_L\supseteq P_Q}\widetilde{\sigma}_L
\Bigr)^{p-1}\rho_Q.
\end{aligned}
\]
Here we used
\(\omega(P_Q)=\widetilde{\omega}_Q\rho_Q\) and
Lemma~\ref{lem:coefficient-estimates}.
For \(1<p<2\), H\"older's inequality with exponents
\(1/(2-p)\) and \(1/(p-1)\) gives
\[
\sum_Q
(\widetilde{\sigma}_Q)^{2-p}
\Bigl(
\sum_{P_L\supseteq P_Q}\widetilde{\sigma}_L
\Bigr)^{p-1}\rho_Q
\le
\Bigl(\sum_Q\widetilde{\sigma}_Q\rho_Q\Bigr)^{2-p}
\Bigl(
\sum_Q\rho_Q
\sum_{P_L\supseteq P_Q}\widetilde{\sigma}_L
\Bigr)^{p-1}.
\]
For \(p=2\), the same bound holds with equality.
Interchanging the sums, enlarging the inner sum, and applying
Lemma~\ref{lem:branch-packing} with exponent \(0\), we obtain
\[
\sum_Q\rho_Q
\sum_{P_L\supseteq P_Q}\widetilde{\sigma}_L
=
\sum_L\widetilde{\sigma}_L
\sum_{P_Q\subseteq P_L}\rho_Q
\lesssim
\eta^{-1}\sum_L\widetilde{\sigma}_L\rho_L.
\]
Since \(\widetilde{\sigma}_Q\rho_Q=\sigma(Q)\), it follows that
\[
\|\mathsf T_{j,R}(\sigma)\|_{L^p_\omega}^p
\lesssim_p
\eta^{-(p-1)}
[\omega,\sigma]_{\Acal_{p,\Gamma_j}(\mu)}
\sum_Q\sigma(Q).
\]

Now suppose that \(p>2\). Grouping the sum by the dyadic generation of \(P_Q\) and applying Lemma~\ref{lemma: principal lemma expectation version} to \((\mathbb R^d,\mu_\omega)\), with exponent \(p\) and \(w=1\), gives
\[
\|\mathsf T_{j,R}(\sigma)\|_{L^p_\omega}^p
\lesssim_p
\sum_Q
\widetilde{\sigma}_Q
\Bigl(
\frac{1}{\omega(P_Q)}
\sum_{P_K\subseteq P_Q}
\widetilde{\sigma}_K\omega(P_K)
\Bigr)^{p-1}
\omega(P_Q).
\]

Using \(\omega(P_K)=\widetilde{\omega}_K\rho_K\) and Lemma~\ref{lem:coefficient-estimates}, then enlarging the \(K\)-sum and applying Lemma~\ref{lem:branch-packing} with exponent
\((p-2)/(p-1)\), we obtain
\[
\begin{aligned}
\frac{1}{\omega(P_Q)}
\sum_{P_K\subseteq P_Q}
\widetilde{\sigma}_K\omega(P_K)
&=
\frac{1}{\omega(P_Q)}
\sum_{P_K\subseteq P_Q}
\widetilde{\sigma}_K\widetilde{\omega}_K\rho_K \\
&\le
\frac{
[\omega,\sigma]_{\Acal_{p,\Gamma_j}(\mu)}^{1/(p-1)}
}{\omega(P_Q)}
\sum_{\substack{K\in\mathcal S_j\\P_K\subseteq P_Q}}
(\widetilde{\omega}_K)^{(p-2)/(p-1)}\rho_K 
\\&
\lesssim_{p} \eta^{-1}
[\omega,\sigma]_{\Acal_{p,\Gamma_j}(\mu)}^{1/(p-1)}
(\widetilde{\omega}_Q)^{-1/(p-1)}.
\end{aligned}
\]
Substituting this estimate into the preceding bound gives
\[
\|\mathsf T_{j,R}(\sigma)\|_{L^p_\omega}^p
\lesssim_p
\eta^{-(p-1)}
[\omega,\sigma]_{\Acal_{p,\Gamma_j}(\mu)}
\sum_Q
\widetilde{\sigma}_Q
(\widetilde{\omega}_Q)^{-1}\omega(P_Q)=
\eta^{-(p-1)}
[\omega,\sigma]_{\Acal_{p,\Gamma_j}(\mu)}
\sum_Q\sigma(Q).
\]

In both cases, passing to the full collection by monotone convergence gives
\begin{equation}
\label{eq:direct-testing-before-packing}
\|\mathsf T_{j,R}(\sigma)\|_{L^p_\omega}^{p}
\lesssim_p
\eta^{-(p-1)}
[\omega,\sigma]_{\Acal_{p,\Gamma_j}(\mu)}
\sum_{\substack{Q\in\mathcal S_j\\Q\subseteq R}}\sigma(Q)\le
\eta^{-p}[\omega,\sigma]_{\Acal_{p,\Gamma_j}(\mu)}[\sigma]_{\AinfD}\,\sigma(R),
\end{equation}
where the last inequality follows from \cite[Theorem~4.2]{nieraeth_weighted_2026}.
Taking \(p\)-th roots proves the direct estimate.

For the dual estimate, note that the reversed relation
\[
\Gamma_j^{-1}:=\{(P_Q,Q):Q\in\mathcal S_j\}
\]
is again a branch, since injectivity and the two branch-nesting
conditions are invariant under interchanging the coordinates.
Both coordinate families remain subfamilies of \(\mathcal S\).
By the definition of the weight characteristics,
\[
[\sigma,\omega]_{\Acal_{p',\Gamma_j^{-1}}(\mu)}
=
[\omega,\sigma]_{\Acal_{p,\Gamma_j}(\mu)}^{p'/p}.
\]
Applying the direct estimate with exponent \(p'\) and with
\(\sigma,\omega\) interchanged, gives the desired estimate.
\end{proof}

\begin{proof}[Proof of Theorem~\ref{thm: finite complexity mixed intro}]
By positivity, it suffices to consider nonnegative \(f\) and \(g\).
Using Proposition~\ref{prop:finite-complexity-branch-decomposition},
we can write
\[
\Gamma\cap(\mathcal S\times\mathcal S)
=\bigsqcup_{j=1}^{M_\kappa}\Gamma_j,
\qquad M_\kappa\lesssim_{d,\kappa}1.
\]
The coordinate families of each branch are subfamilies of
\(\mathcal S\), so Proposition~\ref{prop:branch-mixed-estimate}
applies. For each branch,
\[
\Chat_{\mathcal S,\Gamma_j}(f\sigma,g\omega)
=\int_{\R^d} \mathsf T_j^\sigma f\,g\,\mathrm d\mu_\omega
\le
\|\mathsf T_j^\sigma\|_{L^p_\sigma\to
L^p_\omega}
\|f\|_{L^p_\sigma}
\|g\|_{L^{p'}_\omega}.
\]
Proposition~\ref{prop:branch-mixed-estimate} bounds the operator norm uniformly in \(j\). Summing the displayed estimate over the \(M_\kappa\lesssim_{d,\kappa}1\) branches from Proposition~\ref{prop:finite-complexity-branch-decomposition} proves the theorem.
\end{proof}

\subsection{One weight consequences}
\label{sec:one-weight-diagonal-consequence}
We now specialize to \(\sigma=\omega^{1-p'}\) to obtain
one weight estimates and write
\[
[\omega]_{\Acal_{p,\Gamma}(\mu)}
:=
[\omega,\omega^{1-p'}]_{\Acal_{p,\Gamma}(\mu)}.
\]
The first bound involves both the usual characteristic \([\omega]_{A_p^{\ms{D}}(\mu)}\) and {\([\omega]_{\Acal_{p,\Gamma}(\mu)}\).}
If \((Q,Q)\in\Gamma\) for every \(Q\in\mathcal S\), we obtain a bound involving only {\([\omega]_{\Acal_{p,\Gamma}(\mu)}\)}.
\begin{corollary}
\label{cor:finite-complexity-one-weight}
Let \(1<p<\infty\) {and \(0<\eta<1\). Suppose \(\mathcal S\) is \(\eta\)-sparse, and that \(\Gamma\) has complexity at most \(\kappa\).}
Let \(\omega\) and \(\sigma=\omega^{1-p'}\) be weights.
Then, for \(f\in L^p_\omega\) and \(g\in L^{p'}_\sigma\), 
\[
\begin{aligned}
|\Chat_{\mathcal S,\Gamma}(f,g)|
&\lesssim_{p,d,\kappa,\eta}
[\omega]_{\Acal_{p,\Gamma}(\mu)}^{1/p}
\Bigl(
[\omega]_{A_p^{\ms{D}}(\mu)}^{1/(p(p-1))}
+[\omega]_{A_p^{\ms{D}}(\mu)}^{1/p'}
\Bigr)
\|f\|_{L^p_\omega}\|g\|_{L^{p'}_\sigma}.
\end{aligned}
\]
If \(\{(Q,Q):Q\in\mathcal S\}\subseteq\Gamma\), then
\[
|\Chat_{\mathcal S,\Gamma}(f,g)|
\lesssim_{p,d,\kappa,\eta}
[\omega]_{\Acal_{p,\Gamma}(\mu)}^{\max\{1,1/(p-1)\}}
\|f\|_{L^p_\omega}\|g\|_{L^{p'}_\sigma}.
\]
\end{corollary}
\begin{proof}
By \cite[Proposition 2.2]{hytonen_sharp_2011}, which works verbatim for a general measure $\mu$, we have
\[
[\omega]_{\AinfD}
\leq \ee\, [\omega]_{A_p^{\ms{D}}(\mu)}\qquad\text{and}\qquad
[\sigma]_{\AinfD}
\leq \ee\, [\sigma]_{A_{p'}^{\ms{D}}(\mu)}
=\ee\,[\omega]_{A_p^{\ms{D}}(\mu)}^{1/(p-1)}.
\]
Apply Theorem~\ref{thm: finite complexity mixed intro}
to \(f/\sigma\) and \(g/\omega\), using
\[
\|f/\sigma\|_{L^p_\sigma}
=\|f\|_{L^p_\omega}
\qquad\text{and}\qquad
\|g/\omega\|_{L^{p'}_\omega}
=\|g\|_{L^{p'}_\sigma}.
\]
The first estimate follows.

For the second estimate, assume that
\(\{(Q,Q):Q\in\mathcal S\}\subseteq\Gamma\).
Hölder's inequality and the diagonal terms in the weight characteristic give
\[
\omega(Q)
\le
\eta^{-p}
{[\omega]_{\Acal_{p,\Gamma}(\mu)}}
\,\omega(E_Q)\qquad\text{and}\qquad
\sigma(Q)
\le
\eta^{-p'}
{[\omega]_{\Acal_{p,\Gamma}(\mu)}^{1/(p-1)}}
\,\sigma(E_Q).
\]
Summing over the disjoint sets \(E_Q\) gives the corresponding
weighted packing estimates on \(\mathcal S\).
Using these packing estimates in
\eqref{eq:direct-testing-before-packing} and its counterpart
for the reversed branch gives direct and dual testing bounds
with the respective factors
\[
[\omega]_{\Acal_{p,\Gamma}(\mu)}
^{1/(p-1)}
\qquad\text{and}\qquad
[\omega]_{\Acal_{p,\Gamma}(\mu)}
.
\]
Applying Proposition~\ref{prop:finite-complexity-testing}
on each branch to \(f/\sigma\) and \(g/\omega\),
and summing over the branches, proves the second estimate.
\end{proof}
\section{Application to sparse forms for balanced measures}
\label{sec:cpw-comparison}
We finally specialize to \(d=1\) and return to dyadic intervals. To prove Corollary~\ref{cor:cpw-sharp-quantitative}, we first compare the coefficients in the sparse form in \cite{CPW} with those in Theorem~\ref{thm: finite complexity mixed intro}. We then show that our ordered weight characteristic is equivalent to the \(A_p^N\) characteristic in \cite{CPW}.

Fix \(N\in\mathbb N_0\). For a dyadic interval \(I\), set
\[
m_\mu(I):=\frac{\mu(I_-)\mu(I_+)}{\mu(I)}.
\]
An atomless measure \(\mu\) is called balanced with constant \(\beta\ge1\) if
\[
\beta^{-1}m_\mu(K^{(1)})
\le m_\mu(K)
\le \beta\,m_\mu(K^{(1)})
\]
for every dyadic interval \(K\). For a sparse family \(\mathcal S\subseteq \ms{D}\), the two forms in the sparse domination result in \cite[Theorem 3.3]{CPW} are
\begin{align*}
\Lambda_{\mathcal S}(f,g)
&:=
\sum_{I\in\mathcal S}
\avg{f}{I}\avg{g}{I}\mu(I)\\
\mathcal C_{\mathcal S}^{N}(f,g)
&:=
\sum_{\substack{I,J\in\mathcal S\\
2<\distD(I,J)\le N+2\\
I\cap J=\varnothing}}
\avg{f}{I}\avg{g}{J}
\sqrt{m_\mu(I)m_\mu(J)}.
\end{align*}
Define the separated relation
\[
\Gamma_N^{\mathrm{sep}}
:=
\{(I,J)\in\mathscr D\times\mathscr D:
I\cap J=\varnothing,\ 2<\distD(I,J)\le N+2\}
\]
and put
\[
\Gamma_N
:=
\{(I,I):I\in\mathscr D\}\cup\Gamma_N^{\mathrm{sep}}.
\]
Let \(\mu\) be balanced with constant \(\beta\), and let \(T\) be a Haar shift with coefficients bounded by \(1\) and total complexity at most \(N\). For bounded, compactly supported nonnegative \(f\) and \(g\), the sparse domination theorem \cite[Theorem~3.3]{CPW} gives a sparse family \(\mathcal S\) with a uniform sparsity parameter such that
\[
\lvert\langle Tf,g\rangle\rvert
\lesssim_{N,\beta}
\Lambda_{\mathcal S}(f,g)
+\mathcal C_{\mathcal S}^{N}(f,g).
\]

\begin{lemma}[Coefficient comparison]
\label{lem:cpw-separated-coefficient-comparison}
Assume that \(\mu\) is balanced with constant \(\beta\).
Let \(I,J\in\mathscr D\) satisfy \(\distD(I,J)\le N+2\).
Then
\[
\sqrt{m_\mu(I)m_\mu(J)}
\lesssim_{N,\beta} c_{I,J}.
\]
If, in addition, \(I\cap J=\varnothing\), then
\[
c_{I,J}
\lesssim_{N,\beta}\sqrt{m_\mu(I)m_\mu(J)}.
\]
\end{lemma}

\begin{proof}
Let \(H\) be the minimal common dyadic ancestor of \(I\) and \(J\).
Both intervals lie at most \(N+2\) generations below \(H\).
Iterating the balanced condition along the two chains gives
\[
m_\mu(I)\simeq_{N,\beta}m_\mu(H)
\simeq_{N,\beta}m_\mu(J),
\]
and therefore
\[
\sqrt{m_\mu(I)m_\mu(J)}
\lesssim_{N,\beta}m_\mu(I)
\le \frac14\mu(I).
\]
Interchanging \(I\) and \(J\) gives the same bound with \(\mu(J)\),
and hence the first inequality.

Now suppose that \(I\cap J=\varnothing\).
Let \(I_0\) and \(J_0\) be the children of \(H\) containing
\(I\) and \(J\), respectively. These children are distinct, so
\[
c_{I,J}
\le \min\{\mu(I_0),\mu(J_0)\}
\le
2\,\frac{\mu(I_0)\mu(J_0)}{\mu(I_0)+\mu(J_0)}
=2\,m_\mu(H)\lesssim_{N,\beta}\sqrt{m_\mu(I)m_\mu(J)}.
\]
This proves the second inequality.
\end{proof}

It follows from Lemma \ref{lem:cpw-separated-coefficient-comparison} that  for nonnegative \(f\) and \(g\) we have
\[
\mathcal C_{\mathcal S}^{N}(f,g)
\simeq_{N,\beta}
\Chat_{\mathcal S,\Gamma_N^{\mathrm{sep}}}(f,g).
\]

\begin{lemma}
\label{lem:ordered-cpw-characteristic}
Let \(1<p<\infty\), and let \(\omega\) and
\(\sigma=\omega^{1-p'}\) be weights. If \(\mu\) is balanced with constant
\(\beta\), then
\[
[\omega]_{\Acal_{p,\Gamma_N}(\mu)}
\simeq_{p,N,\beta}
[\omega]_{A_p^N(\mu)}.
\]
\end{lemma}

\begin{proof}
Both characteristics contain the same diagonal contribution
\([\omega]_{A_p^{\ms{D}}(\mu)}\).
For \((I,J)\in\Gamma_N^{\mathrm{sep}}\),
Lemma~\ref{lem:cpw-separated-coefficient-comparison} gives
\[
\frac{c_{I,J}^p}{\mu(I)\mu(J)^{p-1}}
\avg{\omega}{J}\bigl(\avg{\sigma}{I}\bigr)^{p-1}
\simeq_{p,N,\beta}
\frac{(m_\mu(I)m_\mu(J))^{p/2}}
{\mu(I)^p\mu(J)^p}
\omega(J)\sigma(I)^{p-1}.
\]
The right-hand side is a  term in $[\omega]_{A_p^N(\mu)}$ indexed by \((J,I)\). Since interchanging \(I\) and \(J\) preserves \(\Gamma_N^{\mathrm{sep}}\), the contributions from these pairs to the two characteristics are comparable in both directions.

The remaining pairs are nested or siblings. Let \(H\) be their minimal common dyadic ancestor. In either case,
\[
\frac{\sqrt{m_\mu(I)m_\mu(J)}}{\mu(I)\mu(J)}
\lesssim_{N,\beta}\frac{1}{\mu(H)}.
\]
For nested pairs, this follows from Lemma~\ref{lem:cpw-separated-coefficient-comparison}, since \(c_{I,J}=\mu(I)\mu(J)/\mu(H)\).
For siblings, it follows from the balanced condition and \(m_\mu(H)=\mu(I)\mu(J)/\mu(H)\).

Since \(I,J\subseteq H\), the corresponding term in $[\omega]_{A_p^N(\mu)}$ satisfies
\[
\frac{(m_\mu(I)m_\mu(J))^{p/2}\omega(I)\sigma(J)^{p-1}}
{\mu(I)^p\mu(J)^p}
\lesssim_{p,N,\beta}
\frac{\omega(H)\sigma(H)^{p-1}}{\mu(H)^p}
\le [\omega]_{\Acal_{p,\Gamma_N}(\mu)},
\]
where the last inequality uses \((H,H)\in\Gamma_N\).
Taking suprema proves the reverse estimate.
\end{proof}

We conclude with a proof of Corollary~\ref{cor:cpw-sharp-quantitative}.

\begin{proof}[Proof of Corollary~\ref{cor:cpw-sharp-quantitative}]
We first consider nonnegative, bounded, compactly supported \(f\in L^p_\omega\) and \(g\in L^{p'}_\sigma\).
Applying \cite[Theorem~3.3]{CPW} to \(f\) and \(g\), we obtain
\[
\lvert\langle Tf, g\rangle\rvert
\lesssim_{N,\beta}
\Lambda_{\mathcal S}(f,g)
+\mathcal C_{\mathcal S}^{N}(f,g).
\]
Lemma~\ref{lem:cpw-separated-coefficient-comparison} and the definition of
\(\Gamma_N\) give
\[
\Lambda_{\mathcal S}(f,g)
+\mathcal C_{\mathcal S}^{N}(f,g)
\lesssim_{N,\beta}
\Chat_{\mathcal S,\Gamma_N}(f,g).
\]
The relation \(\Gamma_N\) has complexity at most \(N+2\). We have
\[
\|f/\sigma\|_{L^p_\sigma}
=\|f\|_{L^p_\omega}
\qquad\text{and}\qquad
\|g/\omega\|_{L^{p'}_\omega}
=
\|g\|_{L^{p'}_\sigma}.
\]
Applying Theorem~\ref{thm: finite complexity mixed intro} to \(f/\sigma\) and \(g/\omega\), and using Lemma~\ref{lem:ordered-cpw-characteristic}, gives
\[
\begin{aligned}
\lvert\langle Tf, g\rangle\rvert
&\lesssim_{p,N,\beta}
[\omega]_{A_p^N(\mu)}^{1/p}
\bigl([\sigma]_{\AinfD}^{1/p}+[\omega]_{\AinfD}^{1/p'}\bigr)
\|f\|_{L^p_\omega}
\|g\|_{L^{p'}_\sigma}.
\end{aligned}
\]
For general \(f\) and \(g\), decompose their real and imaginary parts into positive and negative parts. Density and duality then give the first operator bound.

Since \(\Gamma_N\) contains the diagonal, Corollary~\ref{cor:finite-complexity-one-weight} applies to \(f\) and \(g\). Combining it with Lemma~\ref{lem:ordered-cpw-characteristic} and using duality gives the second estimate. Since the balanced constant is determined by \(\mu\), the implicit constants may be written as \(\lesssim_{p,N,\mu}\).
\end{proof}
The form \(\Chat_{\mathcal S,\Gamma}\) from Section~\ref{sec:finite-complexity} allows nested pairs.
The following remark shows that extending the two sided comparison in Lemma~\ref{lem:cpw-separated-coefficient-comparison} to all parent child pairs would force a balanced measure to be dyadically doubling.
\begin{remark}
Let \(\mu\) be balanced, and let \(J\) be a child of \(I\) with maximal \(\mu\)-measure. Since \(c_{I,J}=\mu(J)\), balancedness gives
\[
\frac{c_{I,J}}{\sqrt{m_\mu(I)m_\mu(J)}}
\simeq_\beta
\frac{\mu(J)}{m_\mu(I)}
=
\frac{\mu(I)}{\min\{\mu(I_-),\mu(I_+)\}}.
\]
Thus a uniform estimate
\[
c_{I,J}\lesssim\sqrt{m_\mu(I)m_\mu(J)}
\]
on all parent child pairs would bound the last ratio uniformly over \(I\in\mathscr D\). Equivalently, both children would have measure bounded below by a fixed positive multiple of \(\mu(I)\), which is the definition of  dyadic doubling.
For a balanced measure that is not dyadically doubling, the comparison therefore fails already at dyadic distance one. Theorem~\ref{thm: finite complexity mixed intro} uses \(c_{I,J}\) directly and does not require this comparison.
\end{remark}
\section*{AI disclosure statement}
ChatGPT and Codex by OpenAI were used to explore and improve proof strategies for this paper. More specifially:
\begin{itemize}
\item For Theorem~\ref{thm: sparse operators weak type bound}, the authors proposed combining the layer-decomposition strategy from the case \(p=r\) with a weighted change of measure viewpoint. This was developed and refined through interactions with GPT-5.6 Pro, leading to the weighted counting argument used in the final proof.

\item When reviewing an earlier proof of Theorem~\ref{thm: stoch t1 seq test}, GPT-5.5 identified a flaw in the argument and assisted in developing a corrected version.
    \item For Theorem~\ref{thm: finite complexity mixed intro}, the authors had established the general proof strategy, after which GPT-5.5 Pro was used in many iterations to optimize the argument and arrive at its present formulation. Further improvements were obtained through similar interactions with GPT-5.6 Pro after a first manuscript version of the proof had been written.
\end{itemize}
GPT-5.5, GPT-5.6 and GPT-6 were furthermore used to check and improve arguments, to revise parts of the exposition, and to search for relevant references. All mathematical arguments in their final form were written and verified by the authors, who take full responsibility for all claims and arguments in the paper.
\bibliographystyle{abbrv}
\bibliography{references-francisco}

\end{document}